\documentclass[a4paper,11pt,reqno]{amsart}

\usepackage[a4paper,width={16cm},left=2.5cm,bottom=3.5cm, top=3.5cm]{geometry}

\usepackage{color}
\usepackage{hyperref}
\usepackage[english]{babel} 
\usepackage[utf8]{inputenc} 
\usepackage[T1]{fontenc} 
\usepackage{amsmath}
\usepackage{amsfonts}
\usepackage{amssymb}
\usepackage{amsthm}
\usepackage{comment}
\usepackage{lmodern}
\usepackage{mathrsfs}
\usepackage{graphics}
\usepackage{enumerate}
\usepackage{pgf,tikz}
\usepackage{cases}

\numberwithin{equation}{section} 
\numberwithin{table}{section} 

{ \theoremstyle{plain}
\newtheorem{theorem}{Theorem}[section]
\newtheorem{proposition}[theorem]{Proposition}
\newtheorem{lemma}[theorem]{Lemma}
\newtheorem{corollary}[theorem]{Corollary}
\newtheorem{remark}[theorem]{Remark}
\newtheorem{definition}[theorem]{Definition}
}

\newcommand{\cinf}[1]{\textit{C}^{\infty}_{\textrm c}(#1)}

\def\R{\mathbb{R}}

\def\C{\mathbb{C}}
\def\N{\mathbb{N}}

\def\eps{\varepsilon}

\renewcommand{\Re}{\mathop{\mathfrak{Re}}}
\renewcommand{\Im}{\mathop{\mathfrak{Im}}}

\begin{document}

\title[Sharp boundary behavior of eigenvalues]{Sharp boundary behavior of eigenvalues for Aharonov-Bohm operators with varying poles}

\author{Laura Abatangelo}
\address{Laura Abatangelo 
\newline \indent Dipartimento di Matematica e Applicazioni, Università degli Studi di Milano-Bicocca,
\newline \indent  Via Cozzi 55, 20125 Milano, Italy.}
\email{laura.abatangelo@unimib.it}

\author{Veronica Felli}
\address{Veronica Felli 
\newline \indent Dipartimento di Scienza dei Materiali, Università degli Studi di Milano-Bicocca,
\newline \indent Via Cozzi 55, 20125 Milano, Italy.}
\email{veronica.felli@unimib.it}

\author{Benedetta Noris}
\address{Benedetta Noris
\newline \indent Département de Mathématiques, Université Libre de Bruxelles,
\newline \indent CP 214, Boulevard du triomphe, B-1050 Bruxelles, Belgium.}
\email{benedettanoris@gmail.com}

\author{Manon Nys}
\address{Manon Nys 
\newline \indent Dipartimento di Matematica Giuseppe Peano, Università degli Studi di Torino, 
\newline \indent Via Carlo Alberto 10, 10123 Torino, Italy.}
\email{manonys@gmail.com}

\thanks{The authors are partially supported by the project ERC Advanced Grant 2013 n. 339958: ``Complex
Patterns for Strongly Interacting Dynamical Systems --
COMPAT''. L. Abatangelo and V. Felli are partially supported by the
2015 INdAM-GNAMPA research project ``Operatori di Schr\"odinger con
potenziali elettromagnetici singolari: stabilit\`{a} spettrale e stime
di decadimento''. V. Felli is partially supported by PRIN-2012-grant ``Variational and perturbative aspects of nonlinear differential problems''}

\date{May 31, 2016}



\begin{abstract}
  In this paper, we investigate the behavior of the eigenvalues of a
  magnetic Aharonov-Bohm operator with half-integer circulation and
  Dirichlet boundary conditions in a bounded planar domain. We establish a sharp relation between
  the rate of convergence of the eigenvalues as the singular pole is
  approaching a boundary point and the number of nodal lines of the
  eigenfunction of the limiting problem, i.e.\ of the Dirichlet
  Laplacian, ending at that point. The proof relies on the construction of a limit profile
  depending on the direction along which the pole is moving, and on an
  Almgren-type monotonicity argument for magnetic operators.
\end{abstract}

\maketitle

\section{Introduction}\label{sec:introduction}

This paper is concerned with the behavior of the eigenvalues of
Aharonov-Bohm operators in a planar domain with poles approaching the boundary.
For $a = (a_1, a_2) \in \mathbb{R}^2$,
we consider the so-called Aharonov-Bohm magnetic potential with pole
$a$ and circulation $1/2$
\[
A_a(x) = \frac12 \left( \frac{ - (x_2 - a_2)}{(x_1 - a_1)^2 + (x_2 - a_2)^2} , \frac{x_1 - a_1}{(x_1 - a_1)^2 + (x_2 - a_2)^2} \right), \quad x=(x_1,x_2) \in \mathbb{R}^2 \setminus \{a\},
\]
which gives rise to the singular magnetic field $B_a=\mathop{\rm curl}
A_a = \pi \delta_a{\mathbf k}$, where ${\mathbf k}$ is the
unit vector orthogonal to the $x_1x_2$-plane and $\delta_a$ is the
Dirac delta centered at $a$.  Such a magnetic field is generated by an
infinitely long and infinitely thin solenoid intersecting the plane
$x_1x_2$ perpendicularly at $a$.  By Stokes' Theorem, the flux of the
magnetic field through the solenoid cross section is equal (up to the
normalization factor $2\pi$) to the circulation of the vector
potential $A_a$ around the pole $a$, which remains identically equal to
$1/2$.

We consider  the magnetic Schrödinger operator $(i \nabla + A_a)^2$ with
Aharonov-Bohm vector potential $A_a$ 
which acts on functions $u \, : \, \mathbb{R}^2  \to \C$ as
\begin{equation}\label{eq:1}
(i \nabla + A_a)^2 u := - \Delta u + 2 i A_a \cdot \nabla u + |A_a|^2 u,
\end{equation}
and study the properties of the function mapping the position of the pole $a$ to the eigenvalues of the operator \eqref{eq:1} on a bounded domain with homogeneous Dirichlet boundary conditions.

As highlighted in \cite{bonnaillie2014eigenvalues}, the case of
half-integer circulation features a relation between critical
positions of the moving pole and spectral minimal partitions of the
Dirichlet Laplacian.  It was proved in \cite{HHT2009} that the optimal
partition (i.e. the partition of the domain minimizing the largest of
the first eigenvalues on the components) corresponds to the nodal
domain of an eigenfunction of the Dirichlet Laplacian if it has only
points of even multiplicity; the optimal partitions with points of odd
multiplicity are instead related to the
eigenfunctions of the Aharonov-Bohm operator, in the sense that they
can be obtained as 
nodal domains by minimizing a certain eigenvalue of an Aharonov-Bohm Hamiltonian with respect to the number and the position of poles, 
 see \cite{HelfferHO2013}.  We also refer to
\cite{BH,BHH,BHV,HHHO1999,HHOHOO,noris2010nodal} for the study of the eigenfunctions,
their nodal domains and spectral minimal partitions.

The present paper focuses on the behavior of the eigenvalues of the
operator \eqref{eq:1} when the pole $a$ is moving in the domain reaching a
point on the boundary.  Our analysis proceeds by the papers
\cite{abatangelo2015leading,abatangelo2015sharp,bonnaillie2014eigenvalues},
which provide the asymptotic expansion of the eigenvalue function as
the pole is moving in the interior of the domain.  On the other hand,
the study of the case of a pole approaching the boundary was initiated
in \cite{noris2015aharonov}. In this case the limit operator is no
more singular and the magnetic eigenvalues converge to those of the
standard Laplacian. In \cite{noris2015aharonov} the authors predict
the rate of this convergence in relation with the number of nodal
lines that the limit eigenfunction possesses at the limit point.
More precisely, let us denote as $\lambda_N^a$ the $N$-th eigenvalue of the
operator \eqref{eq:1} in a planar domain $\Omega$ with Dirichlet boundary
conditions and as $\lambda_N$ the   $N$-th eigenvalue of the Dirichlet
Laplacian on the same domain;  in \cite{noris2015aharonov} it is proved
that if $\lambda_N$ is simple and the corresponding eigenfunction
$\varphi_N$ has
at a point $b\in\partial \Omega$ a zero of order $j\geq 2$ (so that
$\varphi_N$ has $j-1$ nodal lines ending at $b$) then 
\begin{equation}\label{eq:4}
\lambda_N^a- \lambda_N\leq -C|a-b|^{2j}
\end{equation}
for $a$ moving on a nodal line approaching  $b$, where $C>0$ is a positive constant. In
particular, estimate \eqref{eq:4} implies that, if the pole stays on a nodal
line, then the magnetic eigenvalue is strictly smaller than  the standard
Laplacian's one, thus showing that a diamagnetic-type inequality is not necessarily true for eigenvalues higher than the first one. In the case of the pole approaching a boundary point
$b$ where no nodal lines of $\varphi_N$ end, in \cite{noris2015aharonov} it is proved
that
\begin{equation}\label{eq:5}
\lambda_N^a- \lambda_N\geq C(\mathop{\rm dist}(a,\partial\Omega))^{2}
\end{equation}
as $a\to b$, where $C$ is a positive constant. Estimate \eqref{eq:5} was shown to be sharp in
\cite[Theorem 2.1.15]{Nys-thesis}, where the following exact asymptotics was obtained: 
\begin{equation}\label{eq:7}
\frac{\lambda_N^a- \lambda_N}{(\mathop{\rm
    dist}(a,\partial\Omega))^{2}}\to c(\nabla\varphi_N(b)\cdot\nu)^2
\end{equation}
as $a$ converges to some $b\in\partial\Omega$  where no nodal lines
end, where $c$ is a positive constant.

In the present paper, we describe the asymptotic behavior of the
eigenvalue $\lambda_N^a$ as the pole $a$ approaches a point on the boundary of $\Omega$ moving on
straight lines (not necessarily tangent to nodal lines of the limit
eigenfunction), with the aim of sharpening and generalizing the results in \cite{noris2015aharonov}.
Our main theorem states that, if $\partial\Omega$
is sufficiently smooth, $\lambda_N$ is
simple, and $\varphi_N$ has $j-1$ ($j\in\N$, $j\geq1$) nodal lines
ending at $b\in\partial\Omega$, then the limit of the quotient 
\begin{equation}\label{eq:8}
\frac{\lambda_N-\lambda_N^a}{|a-b|^{2j}},
\end{equation}
as $a$ approches $b$ on a straight line, exists, is finite and
depends continuously on the line direction; furthermore  such a limit is strictly
positive if the line is tangent to a nodal line of $\varphi_N$, while
it is strictly negative if the moving pole direction is 
 in the middle of the tangents to two nodal lines (Theorem
\ref{t:main_straight1}).
This establishes, in particular, that a diamagnetic-type inequality $\lambda_N^a>\lambda_N$ holds for eigenvalues higher than the first one, when $a$ lies in the middle of the tangents to two nodal lines of $\varphi_N$ (or in the middle between a tangent and the boundary). The opposite inequality $\lambda_N^a<\lambda_N$ holds when $a$ belongs to the tangent to a nodal line of $\varphi_N$. Thus, the diamagnetic inequality for this specific operator can be seen as a particular case of Theorem \ref{t:main_straight1}, due to the fact that $\varphi_1$ does not have nodal lines.

Furthermore, we provide a variational characterization of the limit
of the quotient \eqref{eq:8}, by relating it to the minimum of an energy  functional 
associated to an elliptic problem with a crack sloping at the moving
pole direction (Theorem \ref{t:main_straight2}).

Theorem \ref{t:main_straight1} implies that estimate
\eqref{eq:4} is optimal, thus 
generalizing the sharp estimate \eqref{eq:7} to any order of vanishing
of the limit eigenfunction.
Furthermore, our result  answers a question left open in \cite[Remark
1.9]{noris2015aharonov} about the  exact behavior of the eigenvalue
variation $\lambda_N^a- \lambda_N$ as the pole $a$
approaches a boundary point $b$, being $b$ the endpoint of one or more nodal
lines of the limit eigenfunction and $a$ not
belonging to any such nodal line; indeed, as a byproduct of Theorem
\ref{t:main_straight1}, we have that
$\lambda_N^a$ increases  as $a$ is moving 
from a boundary point on the bisector of two nodal lines of the
Dirichlet-Laplacian, or on the bisector of one nodal line and the
boundary, as conjectured in
\cite{noris2015aharonov,Nys-thesis}.

\section{Statement of the main results} \label{sec:results}

Let $\Omega\subset\R^2$ be a bounded, open and simply connected domain. 
We assume that $\Omega\in C^{2,\gamma}$ for some $0<\gamma<1$, and
that 
\[
0\in \partial\Omega.
\] 
Furthermore, it is convenient to suppose that there exists $\bar{R}>0$ such that
\begin{equation} \label{eq:rectifiedbdd}
\Omega \cap D_{\bar{R}}=D_{\bar{R}}^+,
\end{equation}
where $D_{\bar{R}}^+$ is defined as
\[
D_{\bar{R}}^+ := D_{\bar{R}} \cap \R^2_+,
\]
being $D_{\bar{R}}$ the
 open ball of radius $\bar{R}$ centered at $0$ and 
\[
\R^2_+:=\{(x_1,x_2)\in\R^2: x_1 > 0 \}.
\]  
We stress that
this assumption is not restrictive provided that a weight is
considered in the eigenvalue problem.  Starting from a general domain
of class $C^{2,\gamma}$, we can indeed perform a conformal
transformation in order to obtain a new domain satisfying
\eqref{eq:rectifiedbdd}: the counterpart is the appearance of a
conformal weight (real valued) in the new problem, whose regularity is
$C^1(\overline{\Omega})$ thanks to the regularity assumptions on the
domain (see \cite[Theorem 5.2.4]{KrantzBook}). More specifically, the
weight verifies
\begin{equation}\label{eq:q_assumptions}
q(x)\in C^1(\overline{\Omega}), \quad q(x) > 0 \text{ for } x\in \Omega.
\end{equation} 
For more details, we refer the \cite[Section 3]{noris2015aharonov}.

For every $a\in\overline{\Omega}$, we introduce the space $H^{1 ,a}(\Omega,\C)$ as the completion of
\[
\left\{ u \in H^1(\Omega,\C) \cap C^\infty(\Omega,\C)  :  u \text{
    vanishes in a  neighborhood of } a \right\}
\]
with respect to the norm 
\begin{equation}\label{eq:norm_def}
\|u\|_{H^{1,a}(\Omega,\C)}=\left(\left\|\nabla u\right\|^2_{L^2(\Omega,\C^2)} 
+ \|u\|^2_{L^2(\Omega,\C)}
+ \left\|\frac{u}{|x-a|}\right\|^2_{L^2(\Omega,\C)}\right)^{\!\!1/2}.
\end{equation}
For every $a\in\overline\Omega$, we also introduce the space $H^{1 ,a}_0(\Omega,\C)$ 
as the completion of $C^\infty_c(\Omega \setminus \{a\})$
with respect to the norm $\|\cdot\|_{H^{1,a}(\Omega,\C)}$.
In view of the
Hardy-type inequality proved in \cite{LW99} (see 
\eqref{eq:hardy_R2}) and of the Poincaré-type inequality \eqref{eq:Poincare}, 
an equivalent norm in $H^{1,a}_0(\Omega,\C)$ is given by 
\begin{equation}\label{eq:norma}
\|u\|_{H^{1,a}_0(\Omega,\C)}=  \left(\left\|(i\nabla+A_{a}) u\right\|^2_{L^2(\Omega,\C^2)} \right)^{\!\!1/2}.
\end{equation}
As a consequence of the equivalence between norms \eqref{eq:norm_def}
and \eqref{eq:norma}, by gauge invariance it follows that
\begin{equation}\label{eq:equivalent_spaces}
\begin{split}
&\text{if $a\in\partial \Omega$, then the space $H^{1,a}_0(\Omega,\C)$
  coincides with the standard $H^1_0(\Omega,\C)$}\\
&\text{and the norms  \eqref{eq:norm_def}, \eqref{eq:norma}
are therein  equivalent to the Dirichlet norm $\|\nabla u\|_{L^2(\Omega,\C^2)}$. 
}
\end{split}
\end{equation} 
For every $a\in\overline{\Omega}$
and any weight $q(x)$ verifying \eqref{eq:q_assumptions}, we consider the weighted eigenvalue
problem
\begin{equation}\label{eq:eige_equation_a} \tag{$E_a$}
  \begin{cases}
    (i\nabla + A_{a})^2 u = \lambda\, q(x) u,  &\text{in }\Omega,\\
    u = 0, &\text{on }\partial \Omega,
 \end{cases}
\end{equation}
in a weak sense, i.e.\ we say that $\lambda$ is an eigenvalue of \eqref{eq:eige_equation_a} if there exists an eigenfunction $u \in H^{1,a}_0(\Omega, \C) \setminus \{0\}$ such that
\[
\int_{\Omega} (i \nabla + A_a) u \cdot \overline{(i \nabla + A_a) v} \, dx = \lambda \int_{\Omega} q(x) u \overline{v} \, dx, \quad \text{ for all } v \in H^{1,a}_0(\Omega, \C).
\]
From classical spectral theory, $(E_a)$ admits
a diverging sequence of real eigenvalues $\{\lambda_k^a\}_{k\geq 1}$
with finite multiplicity (being each eigenvalue repeated according to
its own multiplicity). To each eigenvalue $\lambda_k^a$ we associate 
an eigenfunction $\varphi_k^a$ suitably normalized (see
\eqref{eq:6} and \eqref{eq:orthonormality}). When $a\in\partial\Omega$, hence in particular when $a=0$,
$\lambda_k^a = \lambda_k$, being $\lambda_k$ the $k$-th weighted
eigenvalue  of the Dirichlet Laplacian (with the same weight $q(x)$);
moreover, 
if \begin{equation*}
\tilde\theta_0:\R^2\setminus\{0\}\to [-\pi,\pi), \quad
\tilde\theta_0(r\cos t,r\sin t)=t\quad\text{if }t\in [-\pi,\pi),
\end{equation*}
is the polar angle centered at $0$ and discontinuous on the half-line
$\{(x_1, 0) \, : \, x_1 < 0 \}$, we have that 
$e^{-\frac i2 \tilde\theta_0}\varphi_k^0 = \varphi_k$ is a  weighted eigenfunction  of
the Laplacian associated to $\lambda_k$, i.e.
\begin{equation}\label{eq:19}
  \begin{cases}
-\Delta\varphi_k=\lambda_k q(x) \varphi_k,  &\text{in }\Omega,\\
    \varphi_k = 0, &\text{on }\partial \Omega.
 \end{cases}
\end{equation}
From \cite[Theorem 1.1]{bonnaillie2014eigenvalues} and \cite[Theorem 1.2]{Lena2015} it is known that, 
for every $k\in \N\setminus\{0\}$, there holds
\begin{equation}\label{eq:conv_autov}
\lambda_k^a \to \lambda_k \quad\text{as }a\to 0. 
\end{equation}
Let us assume that there exists $N\geq 1$ such that
\begin{equation}\label{eq:simple_eigenv_n0}
  \lambda_{N} \quad\text{is simple}.
\end{equation}
We observe that, in view of \cite{Micheletti}, assumption
\eqref{eq:simple_eigenv_n0}
  holds generically with respect to domain (and weight) variations.
Let $\varphi_N\in H^{1}_{0}(\Omega,\C)\setminus\{0\}$ be an
eigenfunction of problem \eqref{eq:19} associated to the eigenvalue
$\lambda_N$ such that 
\begin{equation}\label{eq:9}
\int_\Omega q(x) |\varphi_N(x)|^2\,dx=1.
\end{equation}
From \cite{HW} and \cite{HHT2009} (see also \cite{FF2013})
it is known that
 \begin{equation}\label{eq:zero_of_order_j}
  \varphi_N \text{ has at } 0 \text{ a zero
    of order } j \text{ for some }  j\in \N \setminus \{0\};
\end{equation}
more precisely, there exists $\beta\in\C\setminus\{0\}$ such that 
\begin{equation}\label{eq:asyphi0}
  r^{-j} \varphi_N(r(\cos t,\sin t)) \to 
  \beta
  \psi_j (\cos t,\sin t)
=\beta
 \sin
  \left(j\left(\tfrac\pi2 - t\right)\right),
 \end{equation}
in $C^{1,\tau}([-\frac\pi2,\frac\pi2],\C)$ as $r\to0^+$ for any
$\tau\in (0,1)$. 
Here, for every $j \in \mathbb{N} \setminus \{0\}$, $\psi_j$ is the unique function 
(up to a multiplicative constant) which is harmonic in $\R^2_+$, 
homogeneous of degree $j$ and vanishing on $\partial \R^2_+$, more explicitly
\begin{equation}\label{eq:psi_j}
  \psi_j(r\cos t,r\sin t)= r^{j} \sin
  \left(j\left(\tfrac\pi2 - t\right)\right),\quad 
  r\geq0,\quad t\in\left[-\tfrac\pi2,\tfrac\pi2\right].
\end{equation}
We notice that $\psi_j$ has exactly $j-1$ nodal lines (except for the boundary) 
dividing the $\pi$-angle in equal parts.
Moreover, via a change of gauge, 
\[
\text{the function }e^{\frac i2 \tilde\theta_0 }\psi_j 
\text{ is a distributional solution to }
 (i\nabla + A_0)^2  \big(e^{\frac i2 \tilde\theta_0 }\psi_j\big) =0  \text{ in }\R^2_+.
\]
Let 
\[
\varphi_N^0 = \varphi_N e^{\frac i2 \tilde\theta_0},
\]
so that $\varphi_N^0$ is an
eigenfunction of problem $(E_0)$ associated to the eigenvalue
$\lambda_N$.

As already mentioned, we aim at proving sharp asymptotics for the
convergence \eqref{eq:conv_autov} as the pole $a$ moves along a straight
line up to the origin, see Figure \ref{fig:1}.
 More precisely, we fix 
\[
p\in {\mathbb S}^1_+:=\{(x_1,x_2)\in\R^2:x_1^2+x_2^2=1\text{ and
}x_1>0\},
\]
and study the limit of the quotient
\eqref{eq:8}
as $a=|a|p\to 0$, giving a characterization of such a limit in
terms of the direction $p$, which allows recognizing directions for
which it is nonzero (and possibly positive or negative).  

\begin{figure}
\begin{minipage}[b]{7.8cm}
   \centering
\begin{tikzpicture}[scale=0.5]
  \draw[fill, color = black, opacity=0.1] (0,3) -- (0,-3) to
  [out=270,in=180] (2,-6) 
to [out = 0, in = 200] (4, -4)
 to [out = 20, in = 240] (7, 0) 
  to [out=60, in = 270] (7.5, 2) to [out=90, in= 340] (4, 4) 
  to [out = 160, in = 0] (2,6) to [out= 180, in = 90] (0,3);
  \draw[line width=1pt] (0,-2) to (0,2);
  \draw[dotted,line width=1pt] (0,2) to (0,3);
  \draw[dotted,line width=1pt] (0,-2) to (0,-3);
    \draw[line width=1pt] (0,0) to [out=54, in=-90] (1,3);
  \draw[dotted,line width=1pt] (1,3) to [out=-90, in=-90] (1,4);  
   \draw[line width=1pt] (0,0) to [out=18, in=-108] (2.5,2.5);
 \draw[dotted,line width=1pt] (2.5,2.5) to [out=-108, in=-108] (2.8,3.5); 
   \draw[line width=1pt] (0,0) to [out=-18, in=108] (2.5,-2.5);
\draw[dotted,line width=1pt] (2.5,-2.5) to [out=108, in=108] (2.8,-3.5); 
 \draw[line width=1pt] (0,0) to [out=-54, in=90] (1,-3);
 \draw[dotted,line width=1pt] (1,-3) to [out=-90, in=-90] (1,-4);  
\fill (0,0) circle (3pt) node[above left] {$0$};
\draw[line width=0.5pt]  (0,0) -- (4,1);
\draw[line width=0.5pt]  (0,0) -- (4,0);
\draw[line width=2pt,->] (4,1) -- (3,0.75);
\fill (4,1) circle (3pt) node[above right] {\scriptsize $a$};
\draw[line width=0.5pt,->] (3,0) to [out=95, in=-50] (2.8,0.7);
\node at (3.2,0.3) {\scriptsize $\alpha$};
                              \end{tikzpicture}
\caption{The $j-1$ nodal lines of $\varphi_N$ ending at $0$  
dividing the $\pi$-angle into $j$ equal parts; 
$a$ approaches $0$ along the straight line $a=|a|p$,  $p=(\cos\alpha,\sin\alpha)$.}
\label{fig:1}
\end{minipage}
\
 \begin{minipage}[b]{7.8cm}
\centering
\begin{tikzpicture}[scale=0.5]
  \draw[fill, color = black, opacity=0.1] (0,3) -- (0,-3) to
  [out=270,in=180] (2,-6) 
to [out = 0, in = 200] (4, -4)
 to [out = 20, in = 240] (7, 0) 
  to [out=60, in = 270] (7.5, 2) to [out=90, in= 340] (4, 4) 
  to [out = 160, in = 0] (2,6) to [out= 180, in = 90] (0,3);
    \draw[line width=0.3pt] (0,0) to [out=54, in=-90] (1,3);
   \draw[line width=0.3pt] (0,0) to [out=18, in=-108] (2.5,2.5);
   \draw[line width=0.3pt] (0,0) to [out=-18, in=108] (2.5,-2.5);
 \draw[line width=0.3pt] (0,0) to [out=-54, in=90] (1,-3);
\fill (0,0) circle (3pt) node[above left] {$0$};
   \draw[line width=0.7pt] (0,0) to (1.16,1.6);
\node at (1.392,1.92) {\tiny $+$};
\draw (1.392,1.92) circle (0.25cm);
\node at (1.392,-1.92) {\tiny $+$};
\draw (1.392,-1.92) circle (0.25cm);
    \draw[line width=0.7pt] (0,0) to (1.902,0.61);
\node at (2.292,0.732) {\tiny $+$};
\draw (2.292,0.732) circle (0.25cm);
\node at (2.292,-0.732) {\tiny $+$};
 \draw (2.292,-0.732) circle (0.25cm);
   \draw[line width=0.7pt] (0,0) to (1.16,-1.6);
    \draw[line width=0.7pt] (0,0) to (1.902,-0.61);
    \draw[dashed,line width=0.7pt] (0,0) to (0.61,1.902);
\node at (0.732,2.304) {\tiny $-$};
\draw (0.732,2.304) circle (0.25cm);
\node at (0.732,-2.304) {\tiny $-$};
\draw (0.732,-2.304) circle (0.25cm);
 \draw[dashed,line width=0.7pt] (0,0) to (1.57,1.23);
\node at (1.884,1.476) {\tiny $-$};
\draw (1.884,1.476) circle (0.25cm);
\node at (1.884,-1.476) {\tiny $-$};
\draw (1.884,-1.476) circle (0.25cm);
    \draw[dashed,line width=0.7pt] (0,0) to (0.61,-1.902);
  \draw[dashed,line width=0.7pt] (0,0) to (1.57,-1.23);
  \draw[dashed,line width=0.7pt] (0,0) to (2,0);
\node at (2.4,0) {\tiny $-$};
\draw (2.4,0) circle (0.25cm);
  \end{tikzpicture}
\caption{The sign of the eigenvalue variation $\lambda_N-\lambda_N^a$: 
 positive tangentially to
  nodal lines, negative on bisectors of nodal lines.}
\label{fig:2}
\end{minipage}
\end{figure}

We are now in position to state our first main result.

\begin{theorem}\label{t:main_straight1}
Let
$\Omega\subset\R^2$ be a bounded, open and simply connected
domain of class $C^{2,\gamma}$ for some $0<\gamma<1$, such that
$0\in\partial\Omega$ and  \eqref{eq:rectifiedbdd} holds. Let $q$
satisfy \eqref{eq:q_assumptions}.
Let $N\geq 1$ be such that the $N$-th eigenvalue $\lambda_N$ of
problem \eqref{eq:19} is simple and let 
$\varphi_N\in H^{1}_{0}(\Omega,\C)\setminus\{0\}$ be an eigenfunction 
of \eqref{eq:19} associated to $\lambda_N$ satisfying
\eqref{eq:9}. 
Let $j\in\N\setminus\{0\}$ be the order of vanishing of
$\varphi_N$ at $0$ as in
\eqref{eq:zero_of_order_j}--\eqref{eq:asyphi0}. For $a\in\Omega$, let $\lambda_{N}^a$ be
the $N$-th eigenvalue of problem \eqref{eq:eige_equation_a}.

Then, for every $p\in {\mathbb S}^1_+$, there exists ${\mathfrak
  c}_p\in\R$ such that 
\begin{equation}\label{eq:38}
\frac{\lambda_N-\lambda_N^a}{|a|^{2j}}\to
 |\beta|^2\,{\mathfrak c}_p,\quad\text{as $a=|a|p\to 0$},
\end{equation}
with $\beta\neq0$ being as in \eqref{eq:asyphi0}. Moreover
\begin{enumerate}[\rm (i)]
\item the function $p\mapsto {\mathfrak c}_p$ is continuous on
  ${\mathbb S}^1_+$ and tends to $0$ as $p\to (0,\pm 1)$;
\item  $\mathfrak{c}_p>0$ if the half-line $\{tp:t\geq0\}$ is  tangent
  to a nodal line of $\varphi_N$ in $0$, i.e. if, for some $k=1,\dots,j-1$,
  $p=\big(\cos(\tfrac\pi2-k\frac\pi j),\sin(\tfrac\pi2-k\frac\pi
  j)\big)$;
\item $\mathfrak{c}_p<0$ if the half-line $\{tp:t\geq0\}$ is tangent
  to the bisector of two nodal lines of $\varphi_N$ or to the
  bisector of one nodal line and the boundary, i.e. if, for some
  $k=0,\dots,j-1$,
  $p=\big(\cos(\tfrac\pi2-\frac\pi{2j}(1+2k)),\sin(\tfrac\pi2-\frac\pi{2j}(1+2k))\big)$.
\end{enumerate}
\end{theorem}
The sign properties of $\mathfrak{c}_p$ imply in particular that, as $|a|$ is sufficiently small,
\begin{align*}
&\lambda_N-\lambda_N^a>0 \textrm{ if $a$ is tangent to a nodal line of
  $\varphi_N$ in $0$},\\
&\lambda_N-\lambda_N^a<0 \textrm{ if $a$ lies in the middle of the tangents to two nodal lines of $\varphi_N$ in $0$},
\end{align*}
see Figure \ref{fig:2}, in agreement with the preexisting results
\eqref{eq:4} and \eqref{eq:5}. This fact, together with the continuity
property of $\mathfrak{c}_p$, implies that $\mathfrak{c}_p$ vanishes
at least two times between two nodal lines of $\varphi_N$ in $0$,
and then $\lambda_N-\lambda_N^a=o(|a|^{2j})$ as $a\to 0$ straightly at least
along $2(j-1)$ directions.

\subsection{Variational characterization of the function $p\mapsto
  {\mathfrak c}_p$ and of the limit profile}\label{sec:vari-char-funct}
  
Our second main result is a variational characterization of the function $p\mapsto
  {\mathfrak c}_p$ appearing in Theorem \ref{t:main_straight1}, for
  which the following additional notation is needed.

Let us fix
$\alpha\in\big(-\frac\pi2,\frac\pi2\big)$ and
$p=(\cos\alpha,\sin\alpha)\in {\mathbb S}^1_+$.
We
denote by $\Gamma_p$ the segment joining $0$ to $p$, that is to say
\begin{equation*}
\Gamma_p=\{(r\cos \alpha, r\sin \alpha): r \in
(0,1)\},
\end{equation*}
and define the space $\mathcal{H}_p$ as the completion of 
\[
 \left\{ u\in H^1(\R^2_+\setminus\Gamma_p)  : u=0 \text{ on
   }\partial\R^2_+\text{ and } 
 u=0 \text{ in a neighborhood of }\infty \right\}
\]
with respect to the Dirichlet norm 
\begin{equation}\label{eq:norm_Hp}
\|u\|_{\mathcal H_p}:=\|
\nabla u \|_{L^2(\R^2_+\setminus\Gamma_p)}.
\end{equation} 
From the
Hardy-type inequality for magnetic Sobolev spaces proved in \cite{LW99} (see \eqref{eq:hardy}) and
a change of gauge, it follows that functions in $\mathcal{H}_p$ also
satisfy a Hardy-type inequality, so that $\mathcal{H}_p$ can be
characterized as 
\[
\mathcal{H}_p=\Big\{ u\in L^1_{\rm loc}(\R^2_+):
 \nabla_{\R^2_+\setminus\Gamma_p} u\in L^2(\R^2_+), 
\ \tfrac{u}{|x|}\in L^2(\R^2_+), \text{ and } u=0 \text{
  on }\partial\R^2_+\Big\},
\]
where $\nabla_{\R^2_+\setminus\Gamma_p}u$ denotes the distributional gradient of $u$ in $\R_+^2\setminus \Gamma_p$.

The functions in $\mathcal{H}_p$ may clearly be discontinuous on $\Gamma_p$. 
For this reason, we introduce two trace operators. 
Let us consider the sets $U^+_p=\{(x_1,x_2)\in \R^2_+:x_2>x_1
\tan \alpha\}\cap D_1^+$ and 
$U^-_p=\{(x_1,x_2)\in \R^2_+:x_2<x_1 \tan \alpha\}\cap D_1^+$.
First, for any function $u$ defined in a neighborhood of $U_p^+$,
respectively $U_p^-$, we define the restriction
\begin{equation} \label{eq:Rp}
\mathcal{R}_p^+ (u) = u|_{U^+_p}, \quad \text{respectively} \quad
\mathcal{R}^-_p(u) = u|_{U^-_p}.
\end{equation} 
We observe that, since $\mathcal{R}_p^\pm$ maps $\mathcal{H}_p$ into
$H^1(U^\pm_p)$ continuously,
the trace operators 
\begin{align} \label{eq:traces}
 \gamma_p^{\pm} : \quad &\mathcal{H}_p \longrightarrow H^{1/2}(\Gamma_p) , \,\, \,  u \longmapsto \gamma_p^{\pm}(u) := \mathcal{R}^{\pm}_p(u)|_{\Gamma_p}
\end{align}
are well defined and  continuous from $\mathcal{H}_p$ to $H^{1/2}(\Gamma_p)$.
Furthermore, by Poincaré and Sobolev trace
inequalities, it is easy to verify that the operator norm of  $\gamma_p^{\pm}$ is bounded uniformly
with respect to $p\in{\mathbb S}^1_+$, in the sense that there exists
a constant $L>0$ independent of $p$ such that, recalling \eqref{eq:norm_Hp},
\begin{equation}\label{eq:11}
  \|\gamma_p^{\pm} (u)\|_{H^{1/2}(\Gamma_p)}\leq
  L\|u\|_{\mathcal{H}_p}\quad\text{for all }u\in \mathcal{H}_p.
\end{equation} 
Clearly, for a continuous function $u$, $\gamma_p^+(u) = \gamma_p^-(u)$. 

We will give a variational characterization of the limit of the
quotient \eqref{eq:8} by relating it to the minimum of the functional $J_{p}: \mathcal H_p \to\R$ defined as
\begin{equation} \label{eq:Jbis}
J_{p}(u) =  \frac{1}{2} \int_{\R^2_+\setminus\Gamma_p} |\nabla u|^2
\,dx + j \cos \left( j \left( \tfrac{\pi}{2} - \alpha \right) \right)
\int_{ \Gamma_p}  |x|^{j-1} (\gamma_p^+(u)-\gamma_p^-(u)) \,ds
\end{equation}
on the set  
\begin{equation} \label{eq:spaceKappap}
\mathcal K_p :=\{u\in \mathcal H_p : \ \gamma_p^+(u + \psi_j) + \gamma_p^-(u + \psi_j) = 0 \}.
\end{equation}
The following theorem relates the value ${\mathfrak
  c}_p$ appearing in the limit \eqref{eq:38} with the minimum of $J_p$
over $\mathcal K_p$.

\begin{theorem}\label{t:main_straight2}
The minimum of  $J_p$
over $\mathcal K_p$ is uniquely achieved at a function $w_{p}\in
\mathcal{K}_p$. Furthermore, letting
\begin{equation}\label{eq:m}
{\mathfrak m}_{p}:=\min_{u\in \mathcal K_p}J_{p}(u)=J_{p}(w_{p}),
\end{equation}
we have that 
\[
\mathfrak{c}_p=-2\mathfrak{m}_p,
\]
with $\mathfrak{c}_p$ being as in Theorem \ref{t:main_straight1}.
\end{theorem}

The proofs of Theorems \ref{t:main_straight1} and
\ref{t:main_straight2} rely on the exact determination of the limit of
a suitable blow-up sequence of the eigenfunctions $\varphi_N^a$, in
the spirit of \cite{abatangelo2015leading,abatangelo2015sharp}.  We
emphasize that the boundary case presents some significant additional
difficulties, due to lack of local symmetry and unavailability of
regularity results of the function $a\mapsto \lambda_N^a$ up to the
boundary. The overcoming of these difficulties requires a nontrivial
adaptation of the techniques developed in
\cite{abatangelo2015leading,abatangelo2015sharp} for interior
poles. Being this blow-up result of independent interest, it is
worthwhile to be stated precisely. To this aim, let us define, for
every $\alpha\in[0,2\pi)$ and
$b=(b_1,b_2)=|b|(\cos\alpha,\sin\alpha)\in \R^2\setminus\{0\}$,
\begin{equation*}
\theta_b:\R^2\setminus\{b\}\to [\alpha,\alpha+2\pi)
\quad\text{and}\quad
\theta_0^b:\R^2\setminus\{0\}\to [\alpha,\alpha+2\pi)
\end{equation*}
such that 
\begin{equation}\label{eq:angles}
  \begin{split}
&\theta_b(b+r(\cos t,\sin t))=
  t\quad \text{for all }r>0\text{ and }t\in  [\alpha,\alpha+2\pi),\\
&\theta_0^b(r(\cos t,\sin t))=
  t\quad \text{for all }r>0\text{ and }t\in  [\alpha,\alpha+2\pi). 
\end{split}
\end{equation}
We observe that the difference function $\theta_0^b - \theta_b$ is
regular except for the segment $\{ tb:\ t\in [0,1] \}$.
Moreover, we also define $\theta_0:\R^2\setminus\{0\}\to [0,2\pi)$ as 
\[
\theta_0(\cos t,\sin t) =t\quad\text{for all }
t\in[0,2\pi).
\]
For $a\in \Omega$, let $\varphi_N^a \in H^{1,a}_0(\Omega, \C)$ be an
eigenfunction of \eqref{eq:eige_equation_a} related to the weighted
eigenvalue $\lambda_N^a$, i.e.\ solving
\begin{equation}\label{eq:equation_a}
 \begin{cases}
   (i\nabla + A_a)^2 \varphi_N^a = \lambda_N^a q(x) \,\varphi_N^a,  &\text{in }\Omega,\\
   \varphi_N^a = 0, &\text{on }\partial \Omega,
 \end{cases}
\end{equation} 
and satisfying the normalization conditions 
\begin{equation} \label{eq:6}
  \int_\Omega q(x)|\varphi_N^a(x)|^2\,dx=1 \quad\text{and}\quad 
  \int_\Omega e^{\frac
    i2(\theta_0^a-\theta_a)(x)}q(x)\varphi_N^a(x)\overline{\varphi_N^0(x)}\,dx\in
\R^+.
\end{equation}
The following theorem gives us the behavior of the eigenfunction
$\varphi_N^a$ for $a$ close to the boundary point $0$; more precisely,
it shows that  a homogeneous scaling of order $j$ of $\varphi_N^a$  
along a fixed direction associated to $p\in {\mathbb S}^1_+$ converges
to the limit profile $\Psi_p \in \bigcup_{r>1} H^{1 ,{p}}(D_r^+,\C)$ given by 
\begin{equation}\label{eq:Psip}
 \Psi_{p}:=e^{\frac i2(\theta_{p}-\theta_0^{p} + \tilde \theta_0)}(w_{p}+\psi_j),
\end{equation}
with $w_p$ as in \eqref{eq:m} and $\psi_j$ as in \eqref{eq:psi_j}.

\begin{theorem}\label{thm:blow_up_varphiN}
Let
$\Omega\subset\R^2$ be a bounded, open and simply connected
domain of class $C^{2,\gamma}$ for some $0<\gamma<1$, such that
$0\in\partial\Omega$ and  \eqref{eq:rectifiedbdd} holds. Let $q$
satisfy \eqref{eq:q_assumptions},
 $N\geq 1$ be such that 
 \eqref{eq:simple_eigenv_n0} holds, and 
$j\in\N\setminus\{0\}$ be  the order of vanishing of
a $N$-th eigenfunction $\varphi_N^0$ of $(E_0)$ satisfying
\eqref{eq:9}.
Let 
$\varphi_N^a \in H^{1,a}_0(\Omega, \C)$ solve \eqref{eq:equation_a}--\eqref{eq:6}.
Then, for every $p\in {\mathbb S}^1_+$, 
\[
\frac{\varphi_N^a(|a|x)}{|a|^j} \to \beta \Psi_p \quad \text{ as } a=|a|p \to 0,
\]
in $H^{1,p}(D_R^+, \C)$ for every $R > 1$,  almost everywhere in
$\R^2_+$ and in $C^2_{\rm loc}(\overline{\R^2_+}\setminus\{p\},\C)$, with $\beta\neq0$ as in \eqref{eq:asyphi0}.
\end{theorem}

We notice that the rate of the convergences 
in Theorems \ref{t:main_straight1} and \ref{thm:blow_up_varphiN} is related to the nodal properties 
of the limit eigenfunction, see \eqref{eq:asyphi0}, 
as already highlighted in \cite{abatangelo2015sharp,bonnaillie2014eigenvalues, noris2015aharonov}.
From the results in \cite[Theorem 1.4]{FF2013} we know that the asymptotic behavior in \eqref{eq:asyphi0} 
is in turn related to the so-called Almgren quotient (for a precise definition see \S \ref{sec:mono}). More precisely, 
\begin{equation} \label{eq:NvarphiN0}
\lim_{r \to 0^+} \frac{r\int_{D_r^+} \left( |(i \nabla + A_0) \varphi_N^0|^2 - \lambda_N q(x) |\varphi_N^0|^2 \right) \, dx}{\int_{\partial D_r^+} |\varphi_N^0|^2 \, ds} = j.
\end{equation}

\subsection{Organization of the paper and main ideas}
In \S \ref{sec:limitprofile} we treat the variational characterization of the limit profile described above.
This extends the one obtained in \cite[Proposition 1.6]{noris2015aharonov} 
for the case $j=1$ 
and the one constructed in \cite[Proposition 4.2]{abatangelo2015sharp} for a general $j$ 
when the pole $a$ approaches a fixed point (which in this case lays in
the interior of the domain) tangentially to a nodal line of the limit eigenfunction. 

On one hand, the case $j=1$ is considerably easier because the growth
at infinite of the limit profile is the least possible: this allows
characterizing immediately the limit profile through its Almgren frequency, since the $\liminf$ and the $\limsup$
of the Almgren quotient at infinity are the same.
On the other hand, the construction presented in \cite{abatangelo2015sharp} holds for general $j$, 
but only for $a$ moving  tangentially to a nodal line of the limit
eigenfunction: this restriction 
forces the limit profile to vanish on a half-line, so
that the authors 
are able to construct the limit profile first on a half-plane solving
a minimization problem, then
reflecting and multiplying by a suitable
phase jumping on the half-line. 
Finally, we remark that the sharp estimates obtained in 
\cite{abatangelo2015leading} for $a$ approaching an interior point along a general direction 
don't make use of an explicit construction of the limit profile: in that
case, the
sharp estimate on nodal lines is enough to compute the leading term of
the Taylor expansion of the eigenvalue variation, thanks to symmetry
and periodicity properties of the Fourier coefficients of the limit profile with respect to the
direction.

In the present paper we are dealing with general $j$ as $a$ approaches
a boundary point along a general direction (not even perpendicular to the
boundary of $\Omega$), so that we cannot take advantage of any
remarkable bound for the Almgren quotient nor of any symmetry
property.  
This requires a  completely new approach, based on the construction of
the limit profile by solving an elliptic crack problem 
 prescribing the jump of the
solution along the crack $\Gamma_p$, rather than its value, see
\eqref{eq:w1}--\eqref{eq:w3}.

In \S \ref{sec:mp} we describe the properties of the function
$\mathfrak{m}_p$ defined in \eqref{eq:m}. 

Next we turn to study a suitable blow-up of the eigenfunctions $\varphi_N^a$. Due to the difficulties in proving a priori energy bounds for 
the blow-up sequence
\begin{equation}\label{eq:blow_up_introduction}
\frac{\varphi_N^a(|a|x)}{|a|^j}, 
\end{equation} 
we introduce the following auxiliary blow-up sequence
\begin{equation}\label{eq:blow_up_intermediate}
\tilde{\varphi}_a(x)=
\sqrt{\frac{\bar{K}|a|}{\int_{\partial D_{\bar{K}|a|}} |\varphi_N^a|^2\,ds }} \varphi_N^a(|a|x),
\end{equation} 
for a suitable $\bar{K}>0$. In \S \ref{sec:mono} we take
advantage of the Almgren's frequency function to obtain a priori
bounds on \eqref{eq:blow_up_intermediate}, see \eqref{eq:67}. We
recall that the frequency function in the context of magnetic
operators was first introduced in
\cite{kurata1997} for magnetic potentials in the Kato class and then
extended to Aharonov-Bohm type potentials in \cite{FFT2011}.

\S \ref{sec:upper} and \S \ref{sec:lower} provide
  preliminary upper and lower bounds for the difference
  $\lambda_N-\lambda_N^a$, which are then summarized in Corollary
  \ref{cor:preliminare_sottosopra}. These preliminary estimates are
  obtained by considering suitable competitor functions, and by
  plugging them into the Courant-Fisher minimax characterization of
  eigenvalues. More precisely, to obtain an upper bound for
  $\lambda_N-\lambda_N^a$ we use the Rayleigh quotient for
  $\lambda_N$, and to get a lower bound for $\lambda_N-\lambda_N^a$ we
  use the Rayleigh quotient for $\lambda_N^a$. 

At this first stage, the estimate from above of
$\lambda_N-\lambda_N^a$ is given in terms of the normalization factor
appearing in \eqref{eq:blow_up_intermediate}; in order to  determine the exact asymptotic
behavior of such normalization term, in \S
\ref{sec:energyestimates} we obtain some
energy estimates of the difference between approximating
and limit eigenfunctions after blow-up, exploiting the invertibility of
the differential of the function $F$ defined in
\eqref{def_operatore_F}. As a consequence, in \S \ref{sec:blowup}
we succeed in proving that 
\[
|a|^{-2j-1}\int_{\partial D_{\bar{K}|a|}} |\varphi_N^a|^2\,ds
\]
tends to a positive finite limit depending on $p\in{\mathbb S}^1_+$ as
$a=|a|p\to 0$, and in turn the equivalence of the two blow-up sequences \eqref{eq:blow_up_introduction} and \eqref{eq:blow_up_intermediate}.
This allows us to conclude the proofs of Theorem \ref{thm:blow_up_varphiN} in \S \ref{sec:blowup} and those of Theorems \ref{t:main_straight1}, \ref{t:main_straight2} in \S \ref{sec:end}.

Finally, in the appendix, we recall a Hardy-type
inequality for Aharonov-Bohm operators and some Poincaré-type
inequalities used throughout the paper.

\subsection{Notation} 
\begin{itemize}
\item For  $r>0$ and $a\in\R^2$, $D_r(a)=\{x\in\R^2:|x-a|<r\}$
denotes the disk of center $a$ and radius $r$.
\item For all $r>0$, 
$D_r=D_r(0)$ denotes the disk of center $0$ and
radius $r$.
\item $\R^2_+=\{(x_1,x_2)\in\R^2: x_1>0\}$ and $\R^2_- = \{ (x_1,x_2) \in \R^2 : x_1 < 0\}$.
\item For all $r>0$, $D_r^+=D_r\cap \R^2_+$  denotes the right half-disk of center $0$ and
radius $r$.
\item For $f\in L^\infty(\Omega)$, $\|f\|_\infty=\|f\|_{L^\infty(\Omega)}$.
\end{itemize}

\section{Limit profile} \label{sec:limitprofile}

Keeping in mind the definitions of $\mathcal{R}_p^{\pm}$ \eqref{eq:Rp}
and of $\gamma_p^{\pm}$ \eqref{eq:traces} given in the \S \ref{sec:vari-char-funct},
we introduce the following further notation. For
$p=(\cos\alpha,\sin\alpha)\in {\mathbb S}^1_+$, let
\[
\nu^+_p = (\sin \alpha, -\cos \alpha) \quad\text{and}\quad \nu^-_p = - \nu_p^+
\]
be the normal unit vectors to $\Gamma_p$.
For every $u\in C^1(D_1^+ \setminus \Gamma_p)$ with 
$\mathcal{R}_p^+(u)\in C^1(\overline{U_p^+})$
and $\mathcal{R}_p^-(u)\in C^1(\overline{U_p^-})$, we
define the normal derivatives $\frac{\partial^\pm u}{\partial
  \nu_p^\pm}$ on $\Gamma_p$ respectively as 
\[
\frac{\partial^+ u}{\partial \nu_p^+} := \nabla \mathcal{R}_p^+(u)
\cdot \nu_p^+\bigg|_{\Gamma_p}, \quad \text{ and } \quad
\frac{\partial^- u}{\partial \nu_p^-} := \nabla \mathcal{R}_p^-(u)
\cdot \nu_p^-\bigg|_{\Gamma_p}.
\]
For a function $u$ differentiable in a neighborhood $\Gamma_p$, we get
\begin{equation}\label{eq:normal_derivative_Gamma_regular}
\frac{\partial^+ u}{\partial \nu_p^+} = - \frac{\partial^- u}{\partial
  \nu_p^-}\quad\text{on }\Gamma_p.
\end{equation}
We remark that since $\psi_j$ is differentiable, it verifies
\eqref{eq:normal_derivative_Gamma_regular}, so that
\[
\frac{\partial^+ \psi_j}{\partial \nu_p^+}(r\cos\alpha,r\sin\alpha) = - \frac{\partial^- \psi_j}{\partial \nu_p^-}(r\cos\alpha,r\sin\alpha) = j r^{j-1} \cos \left( j  \left( \tfrac{\pi}{2} - \alpha \right) \right).
\]
Hence the functional $J_{p}: \mathcal H_p \to\R$ defined in \eqref{eq:Jbis} can be equivalently written as
\begin{align*} 
J_{p}(u) &=  \frac{1}{2} \int_{\R^2_+\setminus\Gamma_p} |\nabla u|^2 \,dx + \int_{ \Gamma_p} \frac{\partial^+ \psi_j}{\partial \nu_p^+} (\gamma_p^+(u)-\gamma_p^-(u)) \,ds \\
&=  \frac{1}{2} \int_{\R^2_+\setminus\Gamma_p} |\nabla u|^2 \,dx + \int_{ \Gamma_p} \gamma_p^+(u)\frac{\partial^+ \psi_j}{\partial \nu_p^+}\,ds
 + \int_{ \Gamma_p} \gamma_p^-(u)\frac{\partial^- \psi_j}{\partial \nu_p^-}\,ds.
\end{align*}
In the following lemma we prove that $J_p$ admits a unique minimum
point in the set ${\mathcal K_p}$ defined in \eqref{eq:spaceKappap}.

\begin{lemma}\label{lemma:mp}
The minimum $\mathfrak{m}_p=\min_{\mathcal K_p}J_{p}$ is uniquely achieved at a function $w_{p}\in \mathcal{K}_p$.
Furthermore, $w_{p}$ is the unique solution to the variational problem
\begin{equation}\label{eq:w_euler_lagrange}
  \begin{cases}
w_p\in\mathcal K_p,\\[5pt]
    {\displaystyle{\int_{\R^2_+\setminus\Gamma_p}}}\nabla w_{p}\cdot\nabla\varphi\,dx
    +2 {\displaystyle{\int_{\Gamma_p}}}\dfrac{\partial^+\psi_j}{\partial\nu_p^+}\gamma_p^+(\varphi)\,ds
    =0, \quad\text{for every } \varphi\in \mathcal{K}_p^0,
  \end{cases}
\end{equation}
where
\begin{equation} \label{eq:kappa0}
\mathcal K_p^0 :=\{u\in \mathcal H_p : \  \gamma_p^+(u)+\gamma_p^-(u)=0\}.
\end{equation}
\end{lemma}

\begin{proof}
From \eqref{eq:11} and the continuity of the embedding
$H^{1/2}(\Gamma_p)\hookrightarrow L^2(\Gamma_p)$, we have that there
exists $C>0$ independent of $p\in{\mathbb S}^1_+$ such that, for all
$u\in \mathcal H_p$,
\begin{align*}
\left| \int_{\Gamma_p} \frac{\partial^{\pm} \psi_j}{\partial
    \nu_p^{\pm}} \gamma_p^{\pm} (u) \,ds \right|
&=\left|j \cos \left( j \left( \tfrac{\pi}{2} - \alpha \right) \right)
 \int_{\Gamma_p} |x|^{j-1}\gamma_p^{\pm} (u) \,ds \right|\\
& \leq j  \int_{\Gamma_p} |\gamma_p^{\pm} (u)| \,ds\leq
 j\|\gamma_p^{\pm} (u)\|_{L^2(\Gamma_p)}
\leq C\|\gamma_p^{\pm} (u)\|_{H^{1/2}(\Gamma_p)}
\leq C L \|u\|_{\mathcal H_p}
\end{align*}
and then, from the elementary inequality $ab\leq\frac{a^2}{4\eps}+\eps
b^2$, we deduce that, for every
$\varepsilon>0$, there exists a constant $C_\varepsilon>0$ (depending
on $\eps$ but independent of $p$) such that, for every $u\in\mathcal{H}_p$,
\begin{equation}\label{eq:J_p_coercive}
\left| \int_{\Gamma_p} \frac{\partial^{\pm} \psi_j}{\partial
    \nu_p^{\pm}} \gamma_p^{\pm} (u) \,ds \right| \leq
\varepsilon \|u\|_{\mathcal H_p}^2+ C_{\varepsilon}.
\end{equation}
This implies that $J_p$ is coercive in $\mathcal H_p$. Furthermore
$\mathcal K_p$ is convex and closed by the continuity of the trace operators. Hence, via standard minimization methods, $J_p$ achieves its minimum over $\mathcal{K}_p$ at some function $w_{p} \in
\mathcal{K}_p$. The Euler-Lagrange equation for $w_{p}$ is \eqref{eq:w_euler_lagrange}.

In order to prove uniqueness, let us assume that $w_p$ and $v_p$ solve \eqref{eq:w_euler_lagrange}.
 Then $w_p-v_p\in \mathcal K_p^0$ and, taking the difference between the
 equations \eqref{eq:w_euler_lagrange} for $w_p$ and $v_p$, we have
 that $w_p-v_p$ satisfies 
\begin{equation*}
  \int_{\R^2_+\setminus\Gamma_p}\nabla
  (w_{p}-v_p)\cdot\nabla\varphi\,dx=0,
  \quad\text{for every } \varphi\in \mathcal{K}_p^0,
\end{equation*}
which, choosing $\varphi=w_p-v_p$ yields that
$\int_{\R^2_+\setminus\Gamma_p} |\nabla(w_p-v_p)|^2 \,dx=0$ so that $w_p\equiv v_p$.
\end{proof}

\begin{proposition}\label{prop_tildePsi}
\begin{enumerate}[\rm (i)]
\item For every $p\in{\mathbb S}^1_+$, the function $\Psi_{p}$ defined
  in \eqref{eq:Psip} satisfies the following properties:
\begin{align}
\label{eq:Psip1} & \Psi_{p} \in H^{1 ,{p}}(D_r^+,\C)
\text{ for all }r>1;\\
\label{eq:Psip2} & 
\begin{cases}
  (i\nabla + A_{p})^2 \Psi_{p}=0, &\text{ in } \R^2_+ \text{ in a
    weak } H^{1 ,{p}}-\text{sense},\\
 \Psi_{p}=0, &\text{ on } \partial \R^2_+;
\end{cases}
\\
\label{eq:Psip3} & \int_{\R^2_+\setminus \Gamma_p} \big|(i\nabla + A_{p})(\Psi_{p} -
e^{\frac i2(\theta_{p}-\theta_0^{p} + \tilde \theta_0)}\psi_j )\big|^2\,dx < +\infty;\\
\label{eq:Psip4} & e^{\frac i2(\theta_{p}-\theta_0^{p} + \tilde \theta_0)}
w_{p} =\Psi_p(x)-e^{\frac i2(\theta_{p}-\theta_0^{p} + \tilde \theta_0)}
      \psi_j(x)=O(|x|^{-1}),\quad\text{as }|x|\to+\infty .
\end{align}
\item The function  $\Psi_{p}$ defined in \eqref{eq:Psip} is the
  unique function satisfying \eqref{eq:Psip1}, \eqref{eq:Psip2} and \eqref{eq:Psip3}.
\end{enumerate}
\end{proposition}

\begin{proof}
The fact that $w_{p} \in \mathcal{K}_p$ and the relation
\begin{equation*}
\mathcal R_p^\pm(\theta_p-\theta_0^p)\bigg|_{\Gamma_p}=\pm \pi
\end{equation*}
imply that 
\[
\gamma_p^+(\Psi_p)=\gamma_p^-(\Psi_p).
\]
As a consequence we have that $(i\nabla+A_p)\Psi_p$ (meant as a
distribution in $\R^2_+$) is equal to the $L^2_{\rm loc}(\R^2_+,\C)$-function 
 $i e^{\frac i2(\theta_{p}-\theta_0^{p} + \tilde
   \theta_0)}\nabla_{\R^2_+\setminus\Gamma_p}(w_p+\psi_j)$, thus yielding \eqref{eq:Psip1}.

In order to prove \eqref{eq:Psip2}, we observe that, for any
$\varphi\in \cinf{ \R^2_+ \setminus \{p\}}$, we have that 
$\tilde\varphi:= e^{-\frac i2(\theta_{p}-\theta_0^{p} + \tilde
   \theta_0)}\varphi\in\mathcal K_p^0$ (as defined in
 \eqref{eq:kappa0}). Hence, by \eqref{eq:w_euler_lagrange},
 \begin{align}
 \notag  \int_{\R^2_+}& (i\nabla + A_{p})\Psi_{p}\cdot \overline{(i\nabla + A_{p})\varphi}\,dx
 = \int_{\R^2_+\setminus \Gamma_p} i\,e^{\frac
     i2(\theta_{p}-\theta_0^{p} + \tilde \theta_0)}
   \nabla(w_{p}+\psi_j)
   \cdot \Big(\! -i e^{-\frac i2(\theta_{p}-\theta_0^{p} + \tilde \theta_0)} {\nabla \tilde\varphi}\Big)\,dx\\
 \label{eq:12}  &=  \int_{\R^2_+\setminus \Gamma_p} \nabla(w_{p}+\psi_j)\cdot{\nabla \tilde\varphi}\,dx = -2
   \int_{\Gamma_p}\frac{\partial^+\psi_j}{\partial\nu_p^+}\gamma_p^+({\tilde\varphi})\,ds+
   \int_{\R^2_+\setminus \Gamma_p} \nabla\psi_j\cdot{\nabla
     \tilde\varphi}\,dx.
 \end{align}
Testing the equation $-\Delta\psi_j=0$ by $\tilde\varphi$ and
integrating by parts in $\{(x_1,x_2)\in \R^2_+:x_2<x_1 \tan \alpha\}$
and in $\{(x_1,x_2)\in \R^2_+:x_2>x_1 \tan \alpha\}$ respectively, we obtain that
the right hand side of \eqref{eq:12} is equal to zero. This proves \eqref{eq:Psip2}.
 
Property \eqref{eq:Psip3} is a straightforward consequence of the fact
that $w_p\in\mathcal H_p$. 
To prove \eqref{eq:Psip4}, we observe
  that the Kelvin transform of $w_p$, i.e. the function $\widetilde w_p(x)=w(
\frac{x}{|x|^2})$
belongs to $H^1(D_1^+)$, vanishes in $\partial \R^2_+\cap D_1$, and weakly satisfies
$-\Delta \widetilde w_p=0$ in $D_1^+$. Then from \cite{HW} and \cite{HHT2009} (see also \cite{FF2013})
we deduce that $\widetilde w_p=O(|x|)$ as $|x|\to 0$ and hence 
 $w_p=O(|x|^{-1})$ as $|x|\to +\infty$.

Finally, to prove (ii), let us consider some $\Psi\in\bigcup_{r>1} H^{1 ,{p}}(D_r^+,\C)$
weakly  satisfying 
\begin{equation*}
\begin{cases}
  (i\nabla +  A_{p})^2 \Psi=0, & \text{in $\R^2_+ $}, \\
 \Psi=0, & \text{on }\partial\R^2_+,
\end{cases} 
\end{equation*}
and 
\begin{equation}\label{eq:Psi1}
  \int_{\R^2_+\setminus \Gamma_p} |(i\nabla + A_{p})(\Psi 
  - e^{\frac i2(\theta_{p}-\theta_0^{p} + \tilde \theta_0)}\psi_j )|^2 < +\infty.
\end{equation}
Then the difference $\Phi=\Psi-\Psi_p$ weakly solves
  $(i\nabla + A_{p})^2 \Phi=0$ in $\R^2_+$ and $\Phi=0$ on $\partial
  \R^2_+$. 
 Moreover from
  \eqref{eq:Psip3} and \eqref{eq:Psi1} it follows that
\begin{equation*}
    \int_{\R^2_+} |(i\nabla + A_{p})\Phi(x)|^2dx < +\infty,
\end{equation*}
which, in view of \eqref{eq:Psip2} and \eqref{eq:hardy}, implies that 
$\int_{\R^2_+} |x-p|^{-2}|\Phi(x)|^2\,dx=0$. Hence $\Phi\equiv 0$ in $\R^2_+$ and $\Psi=\Psi_p$.
\end{proof}

\begin{remark}\label{rem:formulazione}
Since $\Psi_p$ solves \eqref{eq:Psip2}, 
classical regularity theory yields that $\Psi_p\in
C^\infty(\overline{\R^2_+}\setminus\{p\},\C)$, whereas from
\cite{FFT2011} it follows that $\Psi_p(x)=O(|x-p|^{1/2})$ and
$\nabla\Psi_p(x)=O(|x-p|^{-1/2})$ as $x\to p$. Therefore we have that $w_p\in
C^\infty(\overline{U^\pm}\setminus\{p\})$ with
$U^+=\{(x_1,x_2)\in \R^2_+:x_2>x_1 \tan \alpha\}$ and 
$U^-=\{(x_1,x_2)\in \R^2_+:x_2<x_1 \tan \alpha\}$, and that
$|\nabla w_p(x)|=O(|x-p|^{-1/2})$.
Then 
\[
\frac{\partial^\pm w_p}{\partial \nu_p^\pm} \in
L^q(\Gamma_p)
\quad\text{and}\quad 
\frac{\partial w_p}{\partial \nu} \in
L^q(\partial D_1\cap\R^2_+)
\quad\text{for all }q<2,
\]
where $\nu(x)=\frac{x}{|x|}$ denotes the unit normal vector to $\partial D_1$.
Using a simple approximation argument and recalling that
$H^{1/2}(\Gamma_p)\hookrightarrow L^q(\Gamma_p)$ for all $q\geq1$, we
obtain the following  formulas for integration by parts:
\begin{equation}
\label{eq:15}
\int_{\R^2_+\setminus\Gamma_p}  \nabla w_{p}\cdot \nabla \varphi \,dx  =
\int_{\Gamma_p} \frac{\partial^+ w_{p}}{\partial \nu_p^+}\gamma_p^+(\varphi)
\,ds
+\int_{\Gamma_p} \frac{\partial^- w_{p}}{\partial \nu_p^-}\gamma_p^-(\varphi)
\,ds, 
\end{equation}
for all $\varphi\in \mathcal H_p$ and  
\begin{equation}
\label{eq:14} 
\int_{D_1^+\setminus\Gamma_p}  \nabla w_{p}\cdot \nabla \varphi \,dx  =
\int_{\partial D_1^+} \frac{\partial w_{p}}{\partial \nu} \varphi\,ds
+\int_{\Gamma_p} \frac{\partial^+ w_{p}}{\partial \nu_p^+}\gamma_p^+(\varphi)
\,ds
+\int_{\Gamma_p} \frac{\partial^- w_{p}}{\partial \nu_p^-}\gamma_p^-(\varphi)
\,ds, 
\end{equation}
for all $\varphi\in H^1(D_1^+\setminus\Gamma_p)$ such that $\varphi=0\text{ on }\partial\R^2_+$. 
\end{remark}

\begin{remark}\label{rem:equazioni_crack}
In view of \eqref{eq:15},  the weak problem \eqref{eq:w_euler_lagrange} solved
by $w_p$ can be reformulated as an elliptic problem with jump
conditions on the internal crack $\Gamma_p$ as follows:
 \begin{numcases}{}
     -\Delta w_{p} = 0, &\text{in }$\R_+^2 \setminus\Gamma_p$,
    \label{eq:w1}\\
\gamma_p^+(w_p + \psi_j) + \gamma_p^-(w_p + \psi_j) = 0,&\text{on } $\Gamma_p$, \label{eq:w2}\\
     \frac{\partial^+ ( w_{p} + \psi_j ) }{\partial \nu_p^+} -
    \frac{\partial^- ( w_{p} + \psi_j ) }{\partial \nu_p^-} = 0, 
    &\text{on } $\Gamma_p$,\label{eq:w3}
  \end{numcases}  
where the equality in \eqref{eq:w3} is meant in the sense of
$L^q(\Gamma_p)$ for any $q<2$ (see Remark \ref{rem:formulazione}) and
hence almost everywhere.  
We refer to \cite{Medkova200599} for elliptic problems in cracked
domains with jumps of the unknown function and its normal derivative
prescribed on the cracks.
\end{remark}

The following result provides a characterization  of ${\mathfrak m}_p$
as a Fourier coefficient of $w_p$. It will be used to relate  ${\mathfrak
  m}_p$ with the optimal
lower/upper bounds for $\lambda_N-\lambda_N^a$, see  Lemmas
\ref{l:lemmakappaR}  and \ref{l:limitfRa}.

\begin{proposition} \label{prop:relationmp}
For every $p\in{\mathbb S}^1_+$, let
\begin{equation}\label{eq:omega_pj_def}
\omega_{p}(r):=\int_{-\pi/2}^{\pi/2} w_{p}(r\cos t,r\sin t)
\sin\left(j\left(\tfrac{\pi}{2}-t\right)\right) \,dt,\quad r\geq1,
\end{equation}
with $w_{p}$ defined in \eqref{eq:m}. Then
\begin{equation*}
\omega_{p}(r)=\omega_{p}(1) r^{-j}\quad \text{for all }r\geq1\quad\text{and}\quad
{\mathfrak m}_{p}=-j\omega_{p}(1).
\end{equation*}
\end{proposition}

\begin{proof}
By direct calculations, since $-\Delta w_{p} = 0$ in $\R_+^2 \setminus D_1^+$,
 we have that $\omega_{p}$ satisfies
\begin{equation*}
-(r^{1+2j}(r^{-j}\omega_{p}(r))')'=0, \quad \text{for } r>1.
\end{equation*}
Hence there exists a constant $C\in\R$ such that
\[
r^{-j}\omega_{p}(r)=\omega_{p}(1)+\frac{C}{2j}\left(1-\frac{1}{r^{2j}}\right),\quad\text{for
all }r\geq 1.
\]
From \eqref{eq:Psip4} it follows that $\omega_{p}(r)=O(r^{-1})$ as $r\to+\infty$.
Hence, letting $r\to+\infty$ in the previous relation, we find $C=-2j \omega_{p}(1)$, so that
$\omega_{p}(r)=\omega_{p}(1) r^{-j}$ for all $r\geq1$. 
By taking the derivative in this relation and in the definition of $\omega_{p}$ \eqref{eq:omega_pj_def}, we obtain
\begin{equation}\label{eq:omega_pj1}
-j \omega_{p}(1)=\int_{\partial D_1^+} \frac{\partial w_{p}}{\partial\nu} \psi_j \,ds.
\end{equation}
Choosing $\varphi=\psi_j$ in \eqref{eq:14} and then replacing \eqref{eq:omega_pj1}, we obtain
\begin{equation}\label{eq:wpj_psij_by_parts1}
\begin{split}
\int_{D_1^+\setminus\Gamma_p}  \nabla w_{p}\cdot \nabla \psi_j \,dx & =
\int_{\partial D_1^+} \frac{\partial w_{p}}{\partial \nu} \psi_j\,ds
+\int_{\Gamma_p} \left( \frac{\partial^+ w_{p}}{\partial \nu_p^+} 
+ \frac{\partial^- w_{p}}{\partial \nu_p^-} \right) \psi_j \,ds \\
& =-j \omega_{p}(1)+\int_{\Gamma_p} \left( \frac{\partial^+ w_{p}}{\partial \nu_p^+} 
+ \frac{\partial^- w_{p}}{\partial \nu_p^-} \right) \psi_j \,ds.
\end{split}
\end{equation}
Testing the equation $-\Delta\psi_j=0$ by $w_p$ and
integrating by parts in $D_1^+\setminus \Gamma_p$, we arrive at
\begin{equation}\label{eq:wpj_psij_by_parts2}
\begin{split}
\int_{D_1^+\setminus\Gamma_p}  \nabla w_{p}\cdot \nabla \psi_j \,dx &=
\int_{\partial D_1^+}  \frac{\partial \psi_j}{\partial\nu} w_{p}\,ds
+\int_{\Gamma_p} \frac{\partial^+\psi_j}{\partial\nu_p^+}(\gamma_p^+(w_{p})-\gamma_p^-(w_{p}))\,ds \\
&=j \omega_{p}(1) +\int_{\Gamma_p} \frac{\partial^+\psi_j}{\partial\nu_p^+}(\gamma_p^+(w_{p})-\gamma_p^-(w_{p}))\,ds,
\end{split}
\end{equation}
where in the last step we used the fact that $\frac{\partial \psi_j}{\partial\nu}=j\psi_j$ on $\partial D_1^+$.
By combining \eqref{eq:wpj_psij_by_parts1} and \eqref{eq:wpj_psij_by_parts2}, we arrive at
\begin{equation}\label{eq:omega_pj2}
j\omega_{p}(1)=\frac{1}{2}\int_{\Gamma_p} \left( \frac{\partial^+ w_{p}}{\partial \nu_p^+} 
+ \frac{\partial^- w_{p}}{\partial \nu_p^-} \right) \psi_j \,ds
-\frac{1}{2} \int_{\Gamma_p} \frac{\partial^+\psi_j}{\partial\nu_p^+}(\gamma_p^+(w_{p})-\gamma_p^-(w_{p}))\,ds.
\end{equation}
On the other hand, taking $\varphi=w_p$ in \eqref{eq:15}, we obtain 
\begin{equation*}
\int_{\R^2_+\setminus\Gamma_p}  |\nabla w_{p}|^2 \,dx  =
\int_{\Gamma_p} \frac{\partial^+ w_{p}}{\partial \nu_p^+}\gamma_p^+(w_p)
\,ds
+\int_{\Gamma_p} \frac{\partial^- w_{p}}{\partial \nu_p^-}\gamma_p^-(w_p)
\,ds, 
\end{equation*}
which, by definition of ${\mathfrak m}_{p}$, yields
\begin{multline}\label{eq:16}
{\mathfrak m}_{p}=J_{p}(w_{p})=
\frac{1}{2} \int_{\Gamma_p}\left(\frac{\partial^+ (w_{p}+\psi_j)}{\partial \nu_p^+}\gamma_p^+(w_{p})+\frac{\partial^- (w_{p}+\psi_j)}{\partial \nu_p^-}\gamma_p^-(w_{p}) \right)\,ds \\
+\frac{1}{2}\int_{\Gamma_p} \frac{\partial^+ \psi_j}{\partial\nu_p^+} (\gamma_p^+(w_{p})-\gamma_p^-(w_{p}))\,ds.
\end{multline}
Moreover \eqref{eq:w2} and \eqref{eq:w3} imply that
\begin{equation}\label{eq:w_bc_consequence}
\frac{\partial^+ (w_{p}+\psi_j)}{\partial \nu_p^+}\gamma_p^+(w_{p}+\psi_j)
+\frac{\partial^- (w_{p}+\psi_j)}{\partial \nu_p^-}\gamma_p^-(w_{p}+\psi_j)=0
\quad\text{on } \Gamma_p.
\end{equation}
Combining \eqref{eq:16} and \eqref{eq:w_bc_consequence} we obtain
\begin{equation}\label{eq:mpj_rewritten2}
{\mathfrak m}_{p}=
-\frac{1}{2} \int_{\Gamma_p}\left(\frac{\partial^+
    (w_{p}+\psi_j)}{\partial \nu_p^+}
  +\frac{\partial^- (w_{p}+\psi_j)}{\partial \nu_p^-} \right)\psi_j\,ds 
+\frac{1}{2}\int_{\Gamma_p} \frac{\partial^+ \psi_j}{\partial\nu_p^+} (\gamma_p^+(w_{p})-\gamma_p^-(w_{p}))\,ds.
\end{equation}
Since $\psi_j$ is regular, it satisfies \eqref{eq:normal_derivative_Gamma_regular}. Then the statement follows by comparing \eqref{eq:omega_pj2} with \eqref{eq:mpj_rewritten2}.
\end{proof}

\section{Properties of $\mathfrak{m}_p$} \label{sec:mp}

In this section we collect some properties of the map $\mathfrak{m}_p$ defined in \eqref{eq:m}.
The next lemma ensures  that $p\mapsto\mathfrak{m}_p$ is not the null
function, by  providing its sign when $p$ belongs either to the bisector of two nodal lines of $\psi_j$, or to one of the nodal lines of $\psi_j$.

\begin{lemma}\label{lemma:mp_signs}
 \begin{enumerate}[\rm (i)]
\item If $p=(\cos\alpha,\sin\alpha)$ with $\alpha = \frac\pi2 - (1+2k)
  \frac{\pi}{2j}$ for some $k = 0, \ldots , j-1$, then
\begin{equation*}
{\mathfrak m}_{p} = \frac{1}{2} \int_{\R_+^2\setminus\Gamma_p} |\nabla w_{p}|^2\,dx>0.
\end{equation*}
\item If $p=(\cos\alpha,\sin\alpha)$ with $\alpha = \frac\pi2
-k\frac{\pi}{j}$ for some $k = 1, \ldots , j-1$, then
\begin{equation*}
{\mathfrak m}_{p} = - \frac{1}{2}\int_{\R_+^2\setminus\Gamma_p} |\nabla w_{p}|^2\,dx<0.
\end{equation*}
\end{enumerate}
\end{lemma}

\begin{proof}
\begin{enumerate}[\rm (i)]
\item If $\alpha = \frac\pi2 -  (1+2k) \frac{\pi}{2j}$ for some $k = 0,
\ldots , j-1$, then $\partial^\pm\psi_j/\partial\nu_p^\pm=0$ on
$\Gamma_p$, so that $J_{p}(u) =  \frac{1}{2} \|u\|_{\mathcal H_p}^2$;
since in this case $0\not\in\mathcal K_p$ (since $\psi_j\not\equiv0$
on $\Gamma_p$), we conclude that ${\mathfrak m}_{p}=\min_{\mathcal K_p}J_p>0$.
\item
In the second case we have that
$\psi_j\equiv0$ and
$\frac{\partial^+ \psi_j}{\partial\nu_p^+ }(r\cos\alpha,r\sin\alpha)=j(-1)^k r^{j-1}$ on
$\Gamma_p$, so that 
\begin{equation}\label{eq:17}
J_{p}(u) =  \frac{1}{2} \int_{\R^2_+\setminus\Gamma_p} |\nabla u|^2
\,dx +2 (-1)^kj\int_{ \Gamma_p}|x|^{j-1} \gamma_p^+(u),\quad\text{for
  all }u\in {\mathcal K}_p.
\end{equation}
From \eqref{eq:17} it follows easily that ${\mathfrak m}_{p}=\min_{\mathcal K_p}J_p<0$.
Furthermore, in this case \eqref{eq:mpj_rewritten2} is reduced to
\[
{\mathfrak m}_{p}=
\frac{1}{2}\int_{\Gamma_p} \frac{\partial^+ \psi_j}{\partial\nu_p^+} (\gamma_p^+(w_{p})-\gamma_p^-(w_{p}))\,ds,
\]
and hence, by definition of $J_p$ and ${\mathfrak m}_p$, 
\[
{\mathfrak m}_{p}=
\frac{1}{2}\left({\mathfrak m}_{p}-\frac12
  \int_{\R^2_+\setminus\Gamma_p} |\nabla w_p|^2 \,dx\right)
\]
which yields that ${\mathfrak m}_{p}=-\frac12
\int_{\R^2_+\setminus\Gamma_p} |\nabla w_p|^2 \,dx$.
\end{enumerate}
\end{proof}

The following proposition establishes the continuity of the map
$p\mapsto{\mathfrak m}_p$.
\begin{proposition}\label{p:continuity_m_p}
The map $p\mapsto \mathfrak{m}_p$ is continuous in ${\mathbb S}^1_+$. Moreover, it can be extended continuously at $p=(0,1)$ and at $p=(0,-1)$ by letting $\mathfrak{m}_{(0,1)}=\mathfrak{m}_{(0,-1)}=0$.
\end{proposition}

\begin{proof}
First we claim that there exists $C>0$ independent of $p$ such that
\begin{equation}\label{eq:continuity_m_p_claim1}
\int_{\R^2_+\setminus \Gamma_p} |\nabla w_p|^2\,dx \leq C 
\quad\text{for every }p\in {\mathbb S}^1_+.
\end{equation}
To prove the claim, we consider a regular cut-off function $\eta$ defined in $\R^2_+$ such that $\eta=1$ in $D_1^+$ and $\eta=0$ in $\R^2_+\setminus D_2^+$. Then $-\eta \psi_j \in \mathcal{K}_p$ for every $p\in \mathbb{S}^1_+$ and
\[
\mathfrak{m}_p \leq J_p(-\eta\psi_j) =\frac{1}{2} \int_{\R^2_+} |\nabla (-\eta\psi_j)|^2 \,dx.
\]
This fact, together with the inequality \eqref{eq:J_p_coercive} applied with $u=w_p$, provides \eqref{eq:continuity_m_p_claim1}.

Let $p_n=(\cos\alpha_n,\sin\alpha_n) \to p=(\cos\alpha,\sin\alpha)$ as $n\to+\infty$, for some $\alpha_n \in(-\pi/2,\pi/2)$, $\alpha \in [-\pi/2,\pi/2]$. We consider the rotation
\begin{equation*}
\mathcal{R}_n=
\left(\begin{array}{cc}
       \cos(\alpha-\alpha_n) & -\sin(\alpha-\alpha_n) \\
       \sin(\alpha-\alpha_n)  & \cos(\alpha-\alpha_n) 
      \end{array}\right). 
\end{equation*}
With a slight abuse of notation, we denote by $w_{p_n}$ the trivial
extension of $w_{p_n}$ in $\R^2$ (extended to 0 in the set
$\R^2_-=\{(x_1,x_2)\in\R^2: \, x_1<0\}$) and we define the rotated functions
\begin{equation*}
\tilde{w}_{n}(\mathcal{R}_n(x))=w_{p_n}(x), \quad x\in \R^2.
\end{equation*}
We define the space $\tilde{\mathcal{H}}_p$ as the completion of 
\[
 \left\{ u\in H^1(\R^2\setminus\Gamma_p)  : u=0 \text{ on
   }(-\infty,0)\times \{0\}\text{ and } 
 u=0 \text{ in a neighborhood of }\infty \right\}
\]
with respect to the norm $\|u\|_{\tilde{\mathcal H}_p}=\|
\nabla u \|_{L^2(\R^2\setminus\Gamma_p)}$.  

We notice that, for all $p\in {\mathbb S}^1_+$,
$\mathcal H_p=\{u\in \tilde{\mathcal H}_p:u=0\text{ a.e.\ in }\R^2_-\}$.
For large $n$, we also define 
\[
\tilde{\mathcal{H}}_{p,n}=
\{u\in \tilde{\mathcal H}_p:u=0\text{ a.e.\ in }H_n^-\},
\]
where $H_n^-=\left\{ (x_1,x_2)\in \R^2: \,
  x_1<-\tan(\alpha-\alpha_n)x_2  \right\}$, and observe that
$\tilde{w}_{n} \in \tilde{\mathcal{H}}_{p,n}$.

Let
$\tilde{\psi}_{j,n}(\mathcal{R}_n(x))=\psi_j(x)$ and $H_n^+=\left\{ (x_1,x_2)\in \R^2: \,
  x_1>-\tan(\alpha-\alpha_n)x_2  \right\}$. 
By \eqref{eq:w2} we have that 
\begin{equation}\label{eq:wtilde_jump}
  \gamma_p^+(\tilde{w}_{n}+\tilde{\psi}_{j,n})+
\gamma_p^-(\tilde{w}_{n}+\tilde{\psi}_{j,n})=0,
\end{equation}
while from  \eqref{eq:w_euler_lagrange} it follows that 
\begin{equation}\label{eq:wtilde_euler_lagrange}
\int_{H_n^+\setminus\Gamma_p}\nabla \tilde{w}_{n}\cdot\nabla\varphi\,dx
+2\int_{\Gamma_p}\frac{\partial^+\tilde{\psi}_{j,n}}{\partial\nu_p^+}\gamma_p^+(\varphi)\,ds =0, 
\end{equation}
for every $\varphi\in \tilde{\mathcal{K}}_{p,n}^0=\{ u\in
\tilde{\mathcal{H}}_{p,n}:\, \gamma_p^+(u)+\gamma_p^-(u)=0 \}$.

Moreover, from \eqref{eq:continuity_m_p_claim1}
it follows that 
\[
\|\tilde{w}_n\|_{\tilde{\mathcal H}_p}^2\leq C,
\]
hence there exist $\tilde{w}_p\in \tilde{\mathcal H}_p$ and  a subsequence $\{\tilde{w}_{n_k}\}_k$ such that 
$\tilde{w}_{n_k} \rightharpoonup \tilde{w}_p$ weakly in
$\tilde{\mathcal H}_p$ and a.e.\ in $\R^2$.
 By  a.e.\ convergence, we have that $\tilde{w}_p=0$ a.e.\ in
$\R^2_-$, hence  
\[
\tilde{w}_p\in \mathcal{H}_p\text{ if $p\in{\mathbb
  S}^1_+$}\quad \text{while}\quad \tilde{w}_p\in {\mathcal D}^{1,2}(\R^2_+)\text{ if
$p=(0,\pm1)$}.
\]
 Moreover,
\eqref{eq:wtilde_jump} and the continuity of the trace embeddings
$\gamma_p^\pm$ defined in 
\eqref{eq:traces}
imply that 
$\gamma_p^+(\tilde{w}_p+\psi_j)+\gamma_p^-(\tilde{w}_p+\psi_j)=0$,
thus yielding  
\begin{equation*}
  \tilde{w}_p\in\mathcal K_p.
\end{equation*}
Recall the definition of $\mathcal K_p^0$ in \eqref{eq:kappa0} and let
\begin{equation}\label{eq:K0_compact}
\varphi\in \mathcal K_p^0 \cap \{ u\in C^\infty(\R^2_+\setminus\Gamma_p): \, \text{supp}(u)\subset\subset \R^2_+\};
\end{equation}
then, for $n$ sufficiently large, $\varphi\in \tilde{\mathcal{K}}_{p,n}^0$
(extended by $0$ 
in $H_n^-$), so that \eqref{eq:wtilde_euler_lagrange} and the weak
$\tilde{\mathcal H}_p$-convergence 
$\tilde{w}_{n_k} \rightharpoonup \tilde{w}_p$
provide
\begin{equation*}
\int_{\R^2_+\setminus\Gamma_p}\nabla \tilde{w}_p \cdot\nabla\varphi\,dx
+2\int_{\Gamma_p}\frac{\partial^+\psi_j}{\partial\nu_p^+}\gamma_p^+(\varphi)\,ds =0.
\end{equation*}
Since the space defined in \eqref{eq:K0_compact} is dense in $\mathcal
K_p^0$, the previous relation holds for every $\varphi\in \mathcal
K_p^0$. Hence $\tilde{w}_p$ satisfies 
\eqref{eq:w_euler_lagrange} if $p\in {\mathbb S}^1_+$, while 
 $\tilde{w}_p$ satisfies $-\Delta \tilde{w}_p = 0$ weakly in $\R^2_+$ if $p=(0,\pm1)$.
Then the uniqueness result proved in Lemma \ref{lemma:mp} implies that
\begin{equation*}
 \tilde{w}_p=w_p \text{ if } p\in {\mathbb S}^1_+
\quad\text{and}\quad  \tilde{w}_p=0 \text{ if } p=(0,\pm1).
\end{equation*}
From Proposition \ref{prop:relationmp} we have that 
\begin{align}
\notag{\mathfrak m}_{p_{n_k}}&=-j\int_{-\pi/2}^{\pi/2} w_{p_{n_k}}(\cos t,\sin t)
\sin\left(j\left(\tfrac{\pi}{2}-t\right)\right) \,dt\\
\label{eq:20}&=
-j\int_{-\pi/2+\alpha-\alpha_{n_k}}^{\pi/2+\alpha-\alpha_{n_k}} \tilde w_{{n_k}}(\cos t,\sin t)
\sin\left(j\left(\tfrac{\pi}{2}-t+\alpha-\alpha_{n_k}\right)\right) \,dt.
\end{align}
The weak
$\tilde{\mathcal H}_p$-convergence 
$\tilde{w}_{n_k} \rightharpoonup \tilde{w}_p$ and continuity of the trace
embedding $\tilde{\mathcal H}_p\hookrightarrow L^2(\partial D_1)$
allow passing to the limit in \eqref{eq:20} thus yielding that 
\[
\lim_{k\to\infty}{\mathfrak m}_{p_{n_k}}=-j\int_{-\pi/2}^{\pi/2} w_p(\cos t,\sin t)
\sin\left(j\left(\tfrac{\pi}{2}-t\right)\right) \,dt={\mathfrak m}_p
\quad\text{if }p\in{\mathbb S}^1_+
\]
and 
\[
\lim_{k\to\infty}{\mathfrak m}_{p_{n_k}}=0\quad\text{if }p=(0,\pm1).
\]
By the Urysohn property, we conclude that $\lim_{n\to\infty}{\mathfrak
  m}_{p_{n}}={\mathfrak m}_p$ if $p\in{\mathbb S}^1_+$ and 
 $\lim_{n\to\infty}{\mathfrak
  m}_{p_{n}}=0$ if $p=(0,\pm1)$. 
\end{proof}

\section{Monotonicity formula and local energy estimates} \label{sec:mono}

For $1 \leq k \leq N$ and $a\in\Omega$, let $\varphi_k^a$ be an
eigenfunction of problem \eqref{eq:eige_equation_a} related to the
eigenvalue $\lambda_k^a$. More precisely, let $\varphi_k^a$ solve
\begin{equation}\label{eq_eigenfunction}
 \begin{cases}
   (i\nabla + A_a)^2 \varphi_k^a = \lambda_k^a q(x) \varphi_k^a,  &\text{in }\Omega,\\
   \varphi_k^a = 0, &\text{on }\partial \Omega,
 \end{cases}
\end{equation}
and satisfy the orthonormality conditions 
\begin{equation}\label{eq:orthonormality}
  \int_\Omega q(x) |\varphi_k^a(x)|^2\,dx=1\quad\text{and}\quad 
  \int_\Omega q(x) \,
  \varphi_k^a(x)\overline{\varphi_\ell^a(x)}\,dx=0\text{ if }k\neq \ell.
\end{equation}
For $k=N$ we choose $\varphi_N^a$ being as in \eqref{eq:6}.
From \eqref{eq:conv_autov}, \eqref{eq:simple_eigenv_n0},
\eqref{eq:9}, 
\eqref{eq:equation_a}, \eqref{eq:6}, \eqref{eq:hardy_R2},
and standard elliptic estimates, we can deduce that
\begin{equation}\label{eq:45}
(i\nabla+A_a)\varphi_N^a\to (i\nabla+A_0)\varphi_N^0\quad\text{in }L^2(\Omega,\C^2)
\end{equation}
and 
\begin{equation}\label{eq:21}
\varphi_N^a\to \varphi_N^0\quad \text{in }H^1(\Omega,\C)\text{ and in
}C^2_{\rm loc}(\Omega,\C).
\end{equation}
The asymptotic behavior of the eigenfunctions $\varphi_k^a$,
for $1 \leq k \leq N$, close to the singular point $a$ was studied in 
\cite[Theorem 1.3]{FFT2011}, \cite[Theorem 2.1]{HHHO1999}; in
particular it is known that there exist coefficients
$c_{a,k},d_{a,k}\in \C$ such that 
\begin{equation*}
\varphi_k^a(a+(r\cos t,r\sin t))=r^{1/2}\frac{e^{it/2}}{\sqrt{\pi}}
\Big(c_{a,k} \cos\left(\tfrac{t}{2}\right)
+d_{a,k} \sin\left(\tfrac{t}{2}\right)\Big) + o(r^{1/2}), \quad \text{as } r\to0.
\end{equation*}
To derive energy estimates for the eigenfunctions $\varphi_k^a$ in
neighborhoods of $0$ with size $|a|$, we use a monotonicity argument based on the study of an Almgren-type frequency
function in the spirit of \cite{almgren}.

\subsection{Almgren-type frequency function}
\begin{definition}
Recall the definition of $\bar{R}$ in \eqref{eq:rectifiedbdd}.
Let $\lambda \in \R$, $b \in \R^2_+$ and $u \in H^{1,b}(D_{\bar R}^+,\C)$, with $u=0$ on $\{x_1=0\}$.
For any  $|b|<r<\bar R$, we define the Almgren-type frequency function as
\[
\mathcal{N}(u,r,\lambda,A_b) = \dfrac{E(u,r,\lambda,A_b)}{H(u,r)},
\]
where
\begin{equation}\label{eq:23}
E(u,r,\lambda,A_b) = \int_{D_r^+} |(i\nabla +A_b)u|^2\,dx - \lambda
\int_{D_r^+} q(x) \, |u|^2\,dx, \quad 
H(u,r) = \dfrac1r \int_{\partial D_r^+} |u|^2\,ds .
\end{equation}
\end{definition}

We first prove that the frequency function of the eigenfunctions
\eqref{eq_eigenfunction} is well defined in a suitable interval.
To this aim, we observe that, since $a\in\Omega\mapsto \lambda_k^a$ admits a
continuous extension on $\overline{\Omega}$ as proved in \cite[Theorem
1.1]{noris2015aharonov}, we have that
\begin{equation} \label{eq:bound-lambda}
  \Lambda:=\sup_{\substack{
      a\in\overline{\Omega}\\1\leq k\leq N}}\lambda_k^a \in(0,+\infty).
\end{equation}

\begin{lemma}\label{lemma:Hpositivo}
\begin{itemize} 
\item[(i)] There exists $0<R_0<\min\{\bar{R}, (2\Lambda\|q\|_\infty)^{-1/2}) \}$ such that 
$H(\varphi_k^a,r)>0$ for all $|a|<R_0$, $r\in(|a|,R_0]$ and $1 \leq k\leq N$.
\item[(ii)] For every $r\in(0,R_0]$, there exist $C_r > 0$ and $\alpha_r \in (0,r)$ such that $H(\varphi_k^a,r) \geq C_r$ for all $|a| < \alpha_r$ and $1 \leq k\leq N$.
\end{itemize}
\end{lemma}

\begin{proof}
We skip the proof of (i), which is very similar to that of \cite[Lemma 5.2]{abatangelo2015sharp}. In order to prove (ii), suppose by contradiction that there exist $0<r\leq R_0$, $a_n\in\Omega$ with $a_n\to0$, $k_n\in \{1,\ldots,N\}$ such that
\begin{equation*}
\lim_{n\to+\infty}H(\varphi_{k_n}^{a_n},r)=0.
\end{equation*}
From \eqref{eq_eigenfunction}, \eqref{eq:orthonormality} and
\eqref{eq:bound-lambda} we deduce that
\begin{equation*}
\int_\Omega |(i\nabla+A_{a_n})\varphi_{k_n}^{a_n}|^2 \,dx = \lambda_{k_n}^{a_n} \leq \Lambda,
\end{equation*}
so that, by the Hardy-type inequality \eqref{eq:hardy_R2}, 
\[
\|\varphi_{k_n}^{a_n}\|_{H^1_0(\Omega,\C)} \leq C,
\]
for a constant $C$ independent of $n$.  Then, along a subsequence,
$\lambda_{k_n}^{a_n}\to\lambda\in \R$ and
$\varphi_{k_n}^{a_n}\to \varphi$ a.e., weakly in $H^1_0(\Omega,\C)$ and
strongly in $L^2(\Omega,\C)$,
for some $\varphi\in H^{1}_0(\Omega,\C)$. From
\eqref{eq:orthonormality} we have that $\int_\Omega
q(x)|\varphi(x)|^2\,dx=1$ and then $\varphi\not\equiv 0$.

By \eqref{eq:equivalent_spaces}, $\varphi \in
H^{1,0}_0(\Omega,\C)$. We notice that $A_{a_n}\varphi_{k_n}^{a_n} \to
A_0 \varphi$ a.e.\ and, in view of \eqref{eq:hardy_R2},
\begin{align*}
\|A_{a_n} \varphi_{k_n}^{a_n}\|_{L^2(\Omega,\C^2)}^2\leq 
4\int_\Omega |(i\nabla+A_{a_n})\varphi_{k_n}^{a_n}|^2 \,dx \leq 4\Lambda.
\end{align*}
Therefore,  up to a subsequence,
$A_{a_n}\varphi_{k_n}^{a_n}\rightharpoonup A_0\varphi$ weakly in
$L^2(\Omega,\C^2)$. Then we can
pass to the limit in \eqref{eq_eigenfunction}, so that
$\lambda=\lambda_{k_0}$ for some $k_0\in
\{1,\ldots,N\}$ and 
\begin{equation}\label{eq:limit_phi_k_n}
(i\nabla+A_0)^2\varphi=\lambda_{k_0} q(x) \varphi \quad \text{in } \Omega.
\end{equation}
Furthermore, by compactness of the trace embedding
$H^{1}(D^+_{r},\C)\hookrightarrow L^2(\partial D^+_{r},\C)$, we have
that
\[
0=\lim_{n\to\infty}\frac1{r}\int_{\partial
  D^+_{r}}|\varphi_{k_n}^{a_n}|^2\,ds=\frac1{r}\int_{\partial
  D^+_{r}}|\varphi|^2\,ds,
\]
which implies that $\varphi=0$ on $\partial D_{r}^+$.  By testing
\eqref{eq:limit_phi_k_n} by $\varphi$ in $D_r^+$, in view of Lemma
\ref{lemma:Poincare}, we obtain that
  \begin{equation*}
    0=\int_{D_r^+}
    \Big(|(i\nabla + A_{0}) \varphi|^2
    - \lambda_{k_0} q(x) |\varphi|^2 \Big)\,dx
    \geq (1-\Lambda \|q\|_\infty r^2)\int_{D^+_r}
    |(i\nabla + A_{0}) \varphi|^2\,dx.
\end{equation*}
Since $r\leq R_0<(2\Lambda \|q\|_\infty)^{-1/2}$, we deduce that $\int_{D_r^+} |(i\nabla +
A_{0}) \varphi|^2\,dx=0$. Lemma \ref{lemma:Poincare} then implies
that $\varphi\equiv 0$ in $D_r^+$. From the unique continuation
principle (see \cite[Corollary 1.4]{FFT2011}) we conclude that
$\varphi\equiv 0$ in $\Omega$, thus giving rise to a contradiction.
\end{proof}

In the following we let 
\begin{equation*}
0<R_0<\min\{\bar{R}, (2\Lambda\|q\|_\infty)^{-1/2}) \}
\end{equation*}
be such that Lemma \ref{lemma:Hpositivo} (i) holds. As a consequence
of Lemma \ref{lemma:Hpositivo} we have that the function $r\mapsto
\mathcal{N}(\varphi_k^a, r,\lambda_k^a,A_a)$ is well defined in the
interval $(|a|,R_0]$  for all $|a|<R_0$ and $1 \leq k\leq N$.

We recall some results proved in \cite{noris2015aharonov}, which will
be used in the sequel.
\begin{lemma}[{\cite[Lemma 5.2]{noris2015aharonov}}]\label{lemma:H_der}
For all $1 \leq k \leq N$ and $a\in\Omega$, let $\varphi_k^a$ be as
in \eqref{eq_eigenfunction}--\eqref{eq:orthonormality}. Then 
\begin{equation}\label{eq:2}
  \frac1 {H(\varphi_k^a,r)}\dfrac{d}{dr} H(\varphi_k^a,r) = \dfrac2r
\mathcal{N}(\varphi_k^a, r,\lambda_k^a,A_a)\quad\text{for all } |a|<
  r< R_0.
\end{equation}
\end{lemma}

\begin{lemma}[{\cite[Lemma 5.3]{noris2015aharonov}}]\label{lemma:H_estimate_below}
Let $1 \leq k \leq N$ and $r_0 \leq R_0$.
If $|a| \leq r_1 < r_2 \leq r_0$, then
\begin{equation*}
\frac{H(\varphi_k^a,r_2)}{H(\varphi_k^a,r_1)} 
\geq e^{-2 \Lambda \|q\|_\infty  r_0^2} \left(\dfrac{r_2}{r_1}\right)^{2}.
\end{equation*}
\end{lemma}

The formula for the derivative of
$E(\varphi_k^a,r,\lambda_k^a,A_a)$ presents some differences with
respect to \cite{noris2015aharonov}, since in \cite{noris2015aharonov}
the integrals in \eqref{eq:23} were taken over half-balls centered at the
projection of $a$ on $\partial\R^2_+$. 
\begin{lemma}\label{l:derE}
Let $p\in{\mathbb S}^1_+$, $1 \leq k \leq N$ and $a = |a|
p$. Then, for all $|a|<r \leq R_0$,
\[
\frac{d}{dr} E(\varphi_k^a,r,\lambda_k^a,A_a) = 
2 \int_{\partial D_r^+} | (i \nabla + A_a) \varphi_k^a \cdot \nu |^{2} \, ds \\
-  \frac{\lambda_k^a}{r} \int_{D_r^+} |\varphi_k^a|^2 (2q+\nabla q\cdot x)\, dx 
- \frac{2}{r} M_k^a ,
\]
where $\nu(x)=\frac{x}{|x|}$ denotes the unit normal vector to $\partial D_r$ and
\begin{equation*}
  M_k^a = \frac14 \Big( a_1 (
  c_{a,k}^2-
  d_{a,k}^2)+2 a_2 c_{a,k} d_{a,k}\Big).
\end{equation*}
Furthermore, there exists $C>0$ depending on $p\in{\mathbb S}^1_+$ such that, for
all $\mu\geq 2$,
\begin{equation}\label{eq:Mka_mu}
\frac{|M_k^a|}{H(\varphi_k^a,\mu |a|)}\leq \frac{C}{\mu^2}.
\end{equation}
\end{lemma}
\begin{proof}
  The expression of $M_k^a$ follows by a Pohozaev-type identity
  in $D_r^+$, proceeding as in \cite[Lemmas
  5.5-5.7]{noris2015aharonov}. Next, in the same spirit as in
  \cite[Lemmas 5.7-5.8]{noris2015aharonov}, we can relate the value
  $M_k^a$ to the function
  $v(y) = \varphi_k^a(|a|y^2 + a)$ defined in $\tilde{\Omega} := \{ y \in \C
  \, : \, |a| y^2 + a \in D_{2|a|}^+ \}$. Such a domain is fixed (with
  respect to $|a|$, but depends on $p$),
 since $a = |a|p$
  is moving on a straight line. Therefore, we proceed exactly in the same way
  as in the proofs therein and obtain a bound depending on~$p$
\[
\frac{|M_k^a|}{H(\varphi_k^a, 2 |a|)} \leq C.
\]
Expression \eqref{eq:Mka_mu} follows from Lemma \ref{lemma:H_estimate_below}.
\end{proof}

\begin{lemma}[{\cite[Lemma 5.11]{noris2015aharonov}}]\label{lemma:N_estimate_above}
Let $1 \leq k \leq N$, $p\in{\mathbb S}^1_+$,
 and $r_0 \leq R_0$.
There exists $c_{r_0,p}$ such that, for all $\mu>2$,
$a=|a|p$ with $|a|<r_0/\mu$, and $\mu|a|\leq r<r_0$, 
\begin{equation*}
e^{\frac{\Lambda \|2q+\nabla q\cdot x\|_\infty}{2 - 2\Lambda {r_0}^2\|q\|_\infty} r^2} 
\left( \mathcal{N}(\varphi_k^a, r,\lambda_k^a,A_a) +1\right)
\leq e^{\frac{\Lambda \|2q+\nabla q\cdot x\|_\infty}{2 - 2\Lambda {r_0}^2\|q\|_\infty} r_0^2}
\left( \mathcal{N}(\varphi_k^a, r_0,\lambda_k^a,A_a) +1\right) +
  \frac{c_{r_0,p}}{\mu^2}.
\end{equation*}
\end{lemma}
\begin{proof}
 The proof proceeds as in \cite[Lemma 5.11]{noris2015aharonov} (see also \cite[Lemma 5.6]{abatangelo2015sharp}), where we can replace $a_1$ with $|a|$ thanks to Lemma \ref{l:derE} above.
\end{proof}

\begin{lemma}\label{l:stimaupN}
  For every $\delta\in(0,1/4)$ and $p\in{\mathbb S}^1_+$ there exist $r_\delta>0$ and
  $K_{\delta,p}>2$ such that, if $\mu\geq
  K_{\delta,p}$, $a=|a|p$ with 
  $|a|<r_\delta / \mu$, and $\mu|a|\leq r<r_\delta$, then
  $\mathcal{N}(\varphi_N^a, r,\lambda_N^a,A_a)\leq j+\delta$.
\end{lemma}
\begin{proof}
  Let $m>0$ be sufficiently small so that $m(2+j+m/2)<1/2$.  By
  assumption \eqref{eq:zero_of_order_j} and by \eqref{eq:NvarphiN0} we
  have that
\[
\lim_{r\to0^+}  \mathcal{N}(\varphi_N^0,r,\lambda_N^0,A_0)=j,
\]
hence we can choose $r_\delta>0$  sufficiently small so that 
\begin{equation*}
  r_\delta<R_0, \quad e^{\frac{\Lambda \|2q+\nabla q\cdot
      x\|_\infty}{2 - 2\Lambda {r_\delta}^2\|q\|_\infty} r_\delta^2} 
\leq 1+\delta m , \quad
\mathcal{N}(\varphi_N^0,r_\delta,\lambda_N^0,A_0)<j+ \delta m.
\end{equation*}
By \eqref{eq:45}--\eqref{eq:21} there exists $\alpha_\delta>0$ such
that $\mathcal{N}(\varphi_N^a,r_\delta,\lambda_N^a,A_a)<j+\delta m$
for every $a$ with $|a|<\alpha_\delta$. We apply Lemma
\ref{lemma:N_estimate_above} with $r_0=r_\delta$ and $k=N$, to deduce
that for every $\mu>2$,
$|a|<\min\{\alpha_\delta,\frac{r_\delta}{\mu}\}$ and
$\mu|a|<r<r_\delta$ it holds
\begin{multline*}
\mathcal{N}(\varphi_N^a,r,\lambda_N^a,A_a)+1\leq (1+\delta m)(1+j+\delta m) +\frac{c_{r_\delta,p}}{\mu^2} \\
\leq 1+j+\delta m(2+j+\delta m) +\frac{c_{r_\delta,p}}{\mu^2} <1+j+\frac{\delta}{2} +\frac{c_{r_\delta,p}}{\mu^2}.
\end{multline*}
To conclude the proof it is sufficient to choose
$K_{\delta,p}>\max\big\{2,
\big(2c_{r_\delta,p}/\delta\big)^{1/2},r_\delta/\alpha_\delta\big\}$.
\end{proof}

\subsection{Local energy estimates}
Let us fix $\delta\in(0,1/4)$ and $p\in{\mathbb S}^1_+$, and  let 
\begin{equation}\label{eq:26}
\bar r=r_\delta>0\quad\text{and}\quad
  \bar K=K_{\delta,p}>2
\end{equation}
be as in Lemma \ref{l:stimaupN}. For all $a\in\Omega$ such that
 $a=|a|p$ and 
 $|a|<\bar r/\bar K$, we denote
\[
H_a=H(\varphi_N^a,\bar K |a|).
\]
As a direct corollary of Lemmas \ref{lemma:H_der}, \ref{lemma:H_estimate_below},
and \ref{l:stimaupN}, we obtain the
following estimates for $H_a$.

\begin{corollary}\label{l:Hbelow}
There exists $C>0$ independent of $|a|$ such that
  \begin{align}
 \label{eq:stima_sotto_radiceH}
 &H_a\geq C |a|^{2(j+\delta)},\quad\text{if }|a|<\min\left\{\frac{\bar
  r}{\bar K},\alpha_{\bar r}\right\},\\
&  \label{eq:stima_sopra_radiceH}   H_a=O(|a|^{2})\quad\text{ as }|a|\to0,
  \end{align}
with $\alpha_{\bar r}$ being as in Lemma \ref{lemma:Hpositivo}, part (ii).
\end{corollary}
\begin{proof}
In view of Lemma \ref{l:stimaupN}, integration of \eqref{eq:2} over the interval $(\bar K |a|,  \bar r)$ yields
\[
H_a \geq H(\varphi_N^a,\bar r) \left(\frac{\bar K |a|}{\bar
    r}\right)^{2(j+\delta)},\quad\text{if }|a|<
\min\left\{\frac{\bar
  r}{\bar K},\alpha_{\bar r}\right\}.
\]
Then Lemma \ref{lemma:Hpositivo} (ii) provides \eqref{eq:stima_sotto_radiceH}.
To prove \eqref{eq:stima_sopra_radiceH} we notice that there exists $C>0$ such that
\[
H_a \leq C H(\varphi_N^a,r_0) |a|^2,
\]
because of Lemma \ref{lemma:H_estimate_below}, and moreover
$\lim_{a\to0}  H(\varphi_N^a,r_0)\leq C$
because of \eqref{eq:21}.
\end{proof}

 From the Poincar\'{e} type Lemmas  \ref{lemma:Poincare} and
\ref{lemma:Poincare_type}, the scaling property of the Almgren-type
frequency function $\mathcal{N}$, 
and 
Lemma \ref{l:stimaupN}, it follows that, for all $R\geq \bar K$, the family of functions 
\begin{equation}\label{eq:67}
\big\{\tilde
 \varphi_a: a=|a|p,\, |a|<\tfrac{\bar r}{R}\big\}
 \text{ is bounded in $H^{1 ,p}(D_R^+,\C)$}
\end{equation}
where 
\begin{equation}\label{def_blowuppate_normalizzate}
 \tilde \varphi_a (x) :=
 \dfrac{\varphi_N^a(|a|x)}{\sqrt{H_a}},
\end{equation}
see \cite[Theorem 5.9]{abatangelo2015sharp} for details in a similar case.
In particular, for all $R\geq \bar K$,  we have that 
\begin{align} \label{eq:46}
& \int_{D^+_{R|a|}} |(i\nabla + A_a)\varphi_N^a|^2dx=O(H_a),\quad\text{as }|a|\to0^+,\\ \nonumber  & \int_{\partial D^+_{R|a|}} |\varphi_N^a|^2dx=O(|a|H_a),\quad\text{as }|a|\to0^+, \\ \label{eq:48} 
& \int_{D^+_{R|a|}} |\varphi_N^a|^2dx=O(|a|^2H_a),\quad\text{as }|a|\to0^+.  
\end{align}
Lemmas \ref{lemma:H_estimate_below} and \ref{lemma:N_estimate_above}
imply the following local energy estimates for the eigenfunctions
$\varphi_k^a$.
\begin{lemma}
  For $1\leq k\leq N$  and $a=|a|p\in\Omega$,
  let $\varphi_k^a \in H^{1,a}_{0}(\Omega,\C)$ be a solution to
  \eqref{eq_eigenfunction}--\eqref{eq:orthonormality}. Let
  $R_0$, $\alpha_{R_0}$ be as in Lemma \ref{lemma:Hpositivo}.  For every $\mu \geq
  \frac{R_0}{\alpha_{R_0}}$, $a=|a|p\in\Omega$ with $|a|<\frac{R_0}{\mu}$, and
  $1\leq k\leq N$, we have that
\begin{align} \label{eq:34}
& \int_{\partial D^+_{\mu|a|}} |\varphi_k^a|^2 \, ds \leq C (\mu|a|)^{3}, \\ \label{eq:35}
& \int_{D^+_{\mu|a|}} |(i\nabla + A_a)\varphi_k^a|^2 \, dx \leq C (\mu|a|)^{2}, \\ \label{eq:36}
& \int_{D^+_{\mu|a|}}|\varphi_k^a|^2 \, dx \leq  C (\mu|a|)^{4},
\end{align}
for some $C > 0$ (depending on $p$).
\end{lemma}
\begin{proof}
  From Lemma \ref{lemma:N_estimate_above}, it follows that, if $\mu >
  2$ and $|a|<\frac{R_0}{\mu}$ then, for all $1 \leq k \leq N$,
\begin{equation}\label{eq:N1}
  \mathcal{N}(\varphi_k^a, \mu|a|,\lambda_k^a,A_a) 
  \leq e^{\frac{\Lambda \|2q+\nabla q\cdot x\|_\infty}{2 - 2\Lambda {R_0}^2\|q\|_\infty} R_0^2}
  \left(\mathcal{N}(\varphi_k^a, R_0,\lambda_k^a,A_a) +1 \right) + \frac{c_{R_0,p}}{\mu^{2}}-1 .
\end{equation}
From \eqref{eq_eigenfunction}, \eqref{eq:orthonormality}, and \eqref{eq:bound-lambda} we
deduce that 
\begin{equation} \label{eq:33}
\int_{D^+_{R_0}} |(i\nabla + A_a)\varphi_k^a|^2 \, dx \leq \int_{\Omega} | (i\nabla + A_a)\varphi_k^a|^2 \, dx = \lambda_k^a \leq \Lambda.
\end{equation}
Therefore, in view of Lemma \ref{lemma:Hpositivo}, if $|a|<\alpha_{R_0}$,
\begin{equation}\label{eq:29}
  \mathcal{N}(\varphi_k^a, R_0,\lambda_k^a,A_a) =
  \frac{\int_{D^+_{R_0}} |(i\nabla + A_a)\varphi_k^a|^2 \, dx -
    \lambda_k^a \int_{D^+_{R_0}} q(x) |\varphi_k^a|^2\,dx}
  {H(\varphi_k^a, R_0)}\leq \frac{\Lambda}{C_{R_0}}.
\end{equation}
Combining \eqref{eq:N1} and \eqref{eq:29} we obtain that, if $\mu \geq
\frac{R_0}{\alpha_{R_0}}$ and $|a| < \frac{R_0}\mu$, then
\[
\int_{D^+_{\mu|a|}} |(i\nabla + A_a)\varphi_k^a|^2 \, dx - \lambda_k^a
\int_{D^+_{\mu|a|}} q(x) |\varphi_k^a|^2 \, dx \leq {\rm const}\,
H(\varphi_k^a, \mu|a|)
\]
for some positive ${\rm const}>0$. Hence, from Lemmas \ref{lemma:Poincare} and \ref{lemma:Poincare_type},
\[
(1- 2 \Lambda \|q\|_{\infty} \mu^2|a|^2) \int_{D^+_{\mu|a|}} |(i\nabla +
A_a)\varphi_k^a|^2\,dx \leq {\rm const} \, H(\varphi_k^a, \mu|a|)
\]
which implies 
\begin{equation}\label{eq:30}
\int_{D^+_{\mu|a|}} |(i\nabla + A_a) \varphi_k^a|^2 \, dx \leq \frac{{\rm const}}{1-2 \Lambda \|q\|_{\infty} R_0^2} H(\varphi_k^a, \mu|a|).
\end{equation}
From Lemma \ref{lemma:H_estimate_below}, it follows that, if
$\mu \geq \frac{R_0}{\alpha_{R_0}}$ and $|a| < \frac{R_0}{\mu}$,
\begin{equation}\label{eq:31}
H(\varphi_k^a,\mu|a|) \leq e^{2 \Lambda \|q\|_{\infty} R_0^2} \left(\dfrac{\mu|a|}{R_0}\right)^{2} H(\varphi_k^a,R_0).
\end{equation}
On the other hand, Lemma \ref{lemma:Poincare_type} and \eqref{eq:33} yield
\begin{equation}\label{eq:32}
H(\varphi_k^a,R_0) \leq \int_{D^+_{R_0}} |( i \nabla + A_a)\varphi_k^a|^2 \, dx \leq \Lambda.
\end{equation}
Estimate \eqref{eq:34} follows combining \eqref{eq:31}, and \eqref{eq:32}, whereas estimate \eqref{eq:35} follows from \eqref{eq:30}, \eqref{eq:31}, and \eqref{eq:32}.  Finally, \eqref{eq:36} can be deduced from \eqref{eq:34}, \eqref{eq:35} and Lemma \ref{lemma:Poincare}.
\end{proof}

\section{Upper bound for $\lambda_N-\lambda_N^a$: the Rayleigh quotient for $\lambda_N$} \label{sec:upper}

Let $R>2$.  For $|a|$ sufficiently small and $1\leq k\leq N$, we define
\begin{equation}\label{eq:vNRa}
 v_{k,R,a}:= 
 \begin{cases}
  v_{k,R,a}^{ext}, &\text{in }\Omega \setminus D^+_{R|a|},\\
  v_{k,R,a}^{int}, &\text{in } D^+_{R|a|},
 \end{cases}
\quad k=1,\ldots,N,
\end{equation}
where 
\begin{equation*}
  v_{k,R,a}^{ext} := e^{\frac{i}{2}(\theta_0^a - \theta_a)}
  \varphi_k^a\quad \text{in }
  \Omega \setminus D^+_{R|a|}, 
\end{equation*}
with $\varphi_k^a$ as in \eqref{eq_eigenfunction}--\eqref{eq:orthonormality} and
$\theta_a,\theta_0^a$ as in \eqref{eq:angles},  so that it solves
\begin{equation*}
 \begin{cases}
   (i\nabla +A_0)^2 v_{k,R,a}^{ext} = \lambda_k^a q \,v_{k,R,a}^{ext}, &\text{in }\Omega \setminus D^+_{R|a|},\\
   v_{k,R,a}^{ext} = e^{\frac{i}{2}(\theta_0^a - \theta_a)} \varphi_k^a
   &\text{on }\partial (\Omega \setminus D^+_{R|a|}),
 \end{cases}
\end{equation*}
whereas $v_{k,R,a}^{int}$ is the unique solution to the problem 
\begin{equation*}
 \begin{cases}
  (i\nabla +A_0)^2 v_{k,R,a}^{int} = 0, &\text{in }D^+_{R|a|},\\
  v_{k,R,a}^{int} = e^{\frac{i}{2}(\theta_0^a - \theta_a)} \varphi_k^a, &\text{on }\partial D^+_{R|a|}.
 \end{cases}
\end{equation*}
It is easy to verify that 
$\mathop{\rm dim}\big(\mathop{\rm span} \{v_{1,R,a},\ldots,v_{N,R,a}\}\big)=N$.

 Arguing as in \cite[Theorem 6.1]{abatangelo2015sharp} and using
estimates \eqref{eq:34}--\eqref{eq:36}, we obtain 
that, for every $R >\max\{2,\frac{R_0}{\alpha_{R_0}}\}$, $a=|a|p\in\Omega$ with  $|a|<\frac{R_0}{R}$, 
and $1\leq k\leq N$, 
\begin{align}
\label{eq:35vint}&\int_{D^+_{R|a|}}|(i\nabla + A_0) v_{k,R,a}^{int}|^2\,dx\leq
\hat C (R|a|)^2,\\
\nonumber &\int_{\partial D^+_{R|a|}}|v_{k,R,a}^{int}|^2\,ds\leq \hat C (R|a|)^{3},\\
\label{eq:36vint}&\int_{D^+_{R|a|}}|v_{k,R,a}^{int}|^2\,dx\leq
\hat C(R|a|)^{4},
\end{align}
for some $\hat C>0$ (depending on $p$ but independent of $|a|$).
For all  $R> \bar K$ and  $a=|a|p\in\Omega$ with
  $|a|$ small, we also define 
\begin{equation}\label{eq:zar}
Z_a^R(x):=\frac{v_{N,R,a}^{int} (|a|x)}{\sqrt{H_a}}.
\end{equation}
As a consequence
of \eqref{eq:67} and of the Dirichlet principle, arguing as in
\cite[Lemma 6.3]{abatangelo2015sharp}, 
we can prove that the family of functions  
\begin{equation}\label{eq:67_zar}
\big\{Z_a^R: a=|a|p, |a|<\tfrac{\bar r}{R}\big\}
 \text{ is bounded in $H^{1 ,0}(D_R^+,\C)$}.
\end{equation}
In particular, for all $R>\bar K$,
 \begin{align}
\label{eq:46zar}&\int_{D^+_{R|a|}} |(i\nabla +
  A_0) v_{N,R,a}^{int}|^2dx=O(H_a),\quad\text{as }|a|\to0^+,\\
\nonumber &\int_{\partial D ^+_{R|a|}} |v_{N,R,a}^{int}|^2dx=O(|a|H_a),\quad\text{as }|a|\to0^+,\\
\label{eq:48zar} &\int_{D^+_{R|a|}} |v_{N,R,a}^{int}|^2dx=O(|a|^2H_a),\quad\text{as }|a|\to0^+.  
\end{align}

\begin{lemma}\label{l:stima_Lambda0_sopra}
Let $p\in{\mathbb S}^1_+$. There exists
  $\tilde R>2$ such that, for all $R>\tilde R$ and  $a=|a|p\in\Omega$ with
  $|a|<\frac{R_0}R$, 
\[
\frac{\lambda_N-\lambda_N^a}{H_a}\leq f_R(a)
\]
where 
\begin{align*}
&f_R(a)=
\int_{D^+_R}|(i\nabla+A_0)Z_a^R|^2\,dx-\int_{D^+_R}|(i\nabla+A_{p})
\tilde \varphi_a|^2\,dx+o(1),\quad\text{as }|a|\to 0^+,\\
\notag &f_R(a)=O(1) ,\quad\text{as }|a|\to 0^+,
\end{align*}
with $\tilde \varphi_a$ and  $Z_a^R$ defined in
\eqref{def_blowuppate_normalizzate} and \eqref{eq:zar} respectively.
In particular $\lambda_N-\lambda_N^a\leq \mathop{\rm const\,}H_a$ as
$a=|a|p\to 0$, for some $\mathop{\rm const\,}>0$ independent of $|a|$.
\end{lemma}

\begin{proof}
  Let us fix 
  $R> \max\{2,\bar K,\frac{R_0}{\alpha_{R_0}}\}$. 
  Let us consider the family of functions $\{\tilde
  v_{k,R,a}\}_{k=1,\dots,N}$ resulting from
  $\{v_{k,R,a}\}_{k=1,\dots,N}$ by a  weighted Gram--Schmidt
  process, that is
\[ 
\tilde v_{k,R,a}:= \dfrac{\hat v_{k,R,a}}{\sqrt{\int_\Omega q|\hat v_{k,R,a}|^2\,dx}}, \quad k=1,\ldots, N, 
\]
where $\hat v_{N,R,a} := v_{N,R,a}$,
\begin{equation*}
\hat v_{k,R,a} := v_{k,R,a} - \sum_{\ell= k+1}^{N} d_{\ell,k}^{R,a} \hat v_{\ell,R,a}, 
\quad \text{for }k=1,\ldots, N-1,
\end{equation*}
and
\[
 d_{\ell,k}^{R,a} :=
\dfrac{\int_\Omega
  q\, v_{k,R,a}  \overline{\hat v_{\ell,R,a}}\,dx}{
\int_\Omega
  q \,|\hat v_{\ell,R,a}|^2\,dx}.
 \]
By constructions, there hold 
\begin{equation}\label{eq:gs}
\int_\Omega q|\tilde v_{k,R,a}|^2\,dx=1\text{ for all }k\quad\text{and}\quad
\int_\Omega q\, \tilde v_{k,R,a}\overline{\tilde v_{\ell,R,a}}\,dx=0\text{ for all }k\neq\ell.
\end{equation}
From \eqref{eq:orthonormality}, 
\eqref{eq:48}, \eqref{eq:36}, \eqref{eq:36vint}, \eqref{eq:48zar},
and an induction argument, we deduce that 
\begin{align}
&\label{eq:18}
\int_\Omega
  q \,|\hat v_{k,R,a}|^2 \,dx =1+O(|a|^{4})
\quad\text{and}\quad 
d_{\ell,k}^{R,a}=O(|a|^{4})\text{ for }\ell\neq k\quad\text{as }
|a|\to 0^+,\\
&\label{eq:13}
\int_\Omega
  q \,|\hat v_{N,R,a}|^2\,dx=
\int_\Omega
  q \,|v_{N,R,a}|^2\,dx
=1+O\big(|a|^{2}H_a\big)
\quad\text{as }|a|\to 0^+,\\
\label{eq:76}
&d_{N,k}^{R,a}= O\big(|a|^{3}\sqrt{H_a}\big)
\quad\text{as }|a|\to 0^+,\quad\text{for all }k<N.
\end{align}
From the classical 
Courant-Fisher minimax characterization of eigenvalues
and \eqref{eq:gs}
it follows that 
\begin{equation*}
  \lambda_N \leq \max_
  {\substack{(\alpha_1,\dots, \alpha_{N})\in \C^{N}\\
      \sum_{k=1}^{{N}}|\alpha_k|^2 =1}}
  \int_{\Omega} \bigg|(i\nabla+A_0) \bigg(\sum_{k=1}^{{N}}\alpha_k
     \tilde v_{k,R,a}\bigg)\bigg|^2 dx,
\end{equation*}
so that 
\begin{equation}\label{eq:107}
\lambda_N-\lambda_N^a\leq \max_
  {\substack{(\alpha_1,\dots, \alpha_{N})\in \C^{N}\\
      \sum_{k=1}^{{N}}|\alpha_k|^2 =1}}\sum_{k,n=1}^{N}
m_{k,n}^{a,R}\alpha_k \overline{\alpha_n},
\end{equation}
where 
\[
m_{k,n}^{a,R}= \int_{\Omega} (i\nabla+A_0)    \tilde v_{k,R,a}\cdot
\overline{(i\nabla+A_0)    \tilde v_{n,R,a}}\, dx
-\lambda_N^a \delta_{kn},
\]
with $\delta_{kn}=1$ if $k=n$ and 
$\delta_{kn}=0$ if $k\neq n$.
From \eqref{eq:13},
\eqref{eq:zar}, and \eqref{def_blowuppate_normalizzate} we deduce that 
\begin{align*}
 \notag m_{N,N}^{a,R} &= \dfrac{\lambda_N^a (1-
\int_\Omega
  q \,|v_{N,R,a}|^2\,dx)}{\int_\Omega
  q \,|v_{N,R,a}|^2\,dx} \\
&\notag\qquad+
  \dfrac{\left( \int_{D^+_{R|a|}} \big| (i\nabla +A_0)v_{N,R,a}^{int}
      \big|^2 dx
      - \int_{D^+_{R|a|}} \big| (i\nabla +A_a)\varphi_N^a \big|^2 dx\right)}{\int_\Omega
  q \,|v_{N,R,a}|^2\,dx} \\
 &= H_a\bigg(\int_{D^+_R}|(i\nabla+A_0)Z_a^R|^2\,dx-\int_{D^+_R}|(i\nabla+A_{p})\tilde \varphi_a|^2\,dx+o(1)\bigg),
\end{align*}
as 
$|a|\to0^+$. 
We observe that, in view of \eqref{eq:67} and \eqref{eq:67_zar}, 
\begin{equation}\label{eq:3}
\int_{D^+_R}|(i\nabla+A_0)Z_a^R|^2\,dx-\int_{D^+_R}|(i\nabla+A_{p})\tilde
\varphi_a|^2\,dx=O(1)
\quad\text{as}\quad |a|\to0^+.
\end{equation}
From \eqref{eq:conv_autov},
 \eqref{eq:18}, \eqref{eq:76}, \eqref{eq:35}, and \eqref{eq:35vint}, we obtain that, if $k<N$, 
\begin{align*}
  m_{k,k}^{a,R}&= -\lambda_N^a 
  +\dfrac{1}{
\int_\Omega
  q \,|\hat v_{k,R,a}|^2 \,dx} \left(
    \lambda_k^a - \int_{D^+_{R|a|}}\!\! \!|(i\nabla
      +A_a)\varphi_k^a|^2dx
    + \int_{D^+_{R|a|}}\! \!\!|(i\nabla +A_0)v_{k,R,a}^{int}|^2dx \right) \\
  &\ + \dfrac{1}{
\int_\Omega
  q \,|\hat v_{k,R,a}|^2\,dx} 
  \int_{\Omega} \bigg|(i\nabla + A_0)
  \Big(\sum_{\ell>k} d_{\ell,k}^{R,a} \hat v_{\ell,R,a} \Big)\bigg|^2dx\\
  & \ -
\dfrac{2}{
\int_\Omega
  q \,|\hat v_{k,R,a}|^2\,dx} 
\Re\bigg(\int_{\Omega} (i\nabla + A_0)v_{k,R,a} \cdot
  \overline{(i\nabla + A_0) \Big( \sum_{\ell>k} d_{\ell,k}^{R,a} \hat v_{\ell,R,a} \Big)}\,dx\bigg)\\
  &=(\lambda_k-\lambda_N)+o(1)\quad\text{as }|a|\to0.
\end{align*}
We observe that from \eqref{eq:simple_eigenv_n0} it follows that
$\lambda_k-\lambda_N<0$ for all $k<N$.

From 
\eqref{eq:46}, \eqref{eq:35}, \eqref{eq:46zar}, and \eqref{eq:35vint},
we deduce that, for all $k<N$,
\begin{align*}
&\left(\int_\Omega q|\hat v_{k,R,a}|^2\,dx\right)^{\!1/2}
\left(\int_\Omega q|\hat v_{N,R,a}|^2\,dx\right)^{\!1/2} m_{k,N}^{a,R}\\
&= \int_{D^+_{R|a|}}\Big((i\nabla+A_0)
    v_{k,R,a}^{int}\cdot \overline{(i\nabla+A_0) v_{N,R,a}^{int} }-
    (i\nabla+A_a) \varphi_k^a\cdot
    \overline{(i\nabla+A_a) \varphi_N^a}\Big)\,dx \\
 &\quad - \int_{\Omega} (i\nabla+A_0)
    \Big(\sum_{\ell>k}d_{\ell,k}^{R,a} \hat v_{\ell,R,a}\Big) \cdot
    \overline{(i\nabla+A_0) v_{N,R,a} }\,dx
=O\Big(|a|  \sqrt{H_a}\Big),
\end{align*}
so that, by \eqref{eq:18} and \eqref{eq:13},
\[
  m_{k,N}^{a,R}=O\Big(|a|  \sqrt{H_a}\Big)\quad\text{and}\quad 
m_{N,k}^{a,R}=\overline{ m_{k,N}^{a,R}}=O\Big(|a| \sqrt{H_a}\Big)
\]
as 
$|a|\to 0^+$.
In a similar way, from \eqref{eq:35} and \eqref{eq:35vint} we can
deduce that, for all $k,n<N$ with $k\neq n$,
\[
  m_{k,n}^{a,R}=O(|a|^{2}) \quad\text{as }|a|\to0.
\]
Thanks to Corollary \ref{l:Hbelow} we can apply 
\cite[Lemma 6.1]{abatangelo2015sharp} to conclude that 
\[
  \max_
  {\substack{(\alpha_1,\dots, \alpha_{N})\in \C^{N}\\
      \sum_{k=1}^{{N}}|\alpha_k|^2 =1}}\sum_{j,n=1}^{N}
  m_{k,n}^{a,R}\alpha_k \overline{\alpha_n}\\ =H_a\bigg(
  \int_{D^+_R}|(i\nabla+A_0)Z_a^R|^2\,dx-\int_{D^+_R}|(i\nabla+A_p)\tilde
  \varphi_a|^2\,dx+o(1)\bigg)
\]
as $|a|\to 0^+$. The conclusion then follows from \eqref{eq:107} and \eqref{eq:3}.
\end{proof}

\section{Lower bound for $\lambda_N-\lambda_N^a$: the Rayleigh quotient for $\lambda_N^a$} \label{sec:lower}

For $R>2$, $1\leq k\leq N$, and $|a|$ sufficiently small we define
\[
 w_{k,R,a}:= 
 \begin{cases}
  w_{k,R,a}^{ext}, &\text{in }\Omega \setminus D^+_{R|a|},\\
  w_{k,R,a}^{int}, &\text{in } D^+_{R|a|},
 \end{cases}
\quad k=1,\ldots,N,
\]
where $w_{k,R,a}^{ext} := e^{\frac{i}{2}( \theta_a-\theta_0^a)}
  \varphi_k^0\quad \text{in }
  \Omega \setminus D^+_{R|a|}$
solves
\begin{equation*}
 \begin{cases}
   (i\nabla +A_a)^2 w_{k,R,a}^{ext} = \lambda_k q \,w_{k,R,a}^{ext}, &\text{in }\Omega \setminus D^+_{R|a|},\\
   w_{k,R,a}^{ext} = e^{\frac{i}{2}(\theta_a-\theta_0^a)} \varphi_k^0,
   &\text{on }\partial (\Omega \setminus D^+_{R|a|}),
 \end{cases}
\end{equation*}
whereas $w_{k,R,a}^{int}$ is the unique solution to the problem 
\begin{equation*}
 \begin{cases}
  (i\nabla +A_a)^2 w_{k,R,a}^{int} = 0, &\text{in }D^+_{R|a|},\\
  w_{k,R,a}^{int} = e^{\frac{i}{2}(\theta_a-\theta_0^a)} \varphi_k^0, &\text{on }\partial D^+_{R|a|}.
 \end{cases}
\end{equation*}
From \eqref{eq:orthonormality} it follows easily that 
$\mathop{\rm dim}\big(\mathop{\rm span} \{w_{1,R,a},\ldots,w_{N,R,a}\}\big)=N$.
From \cite{HW} and \cite{HHT2009} (see also \cite{FF2013}) we have that
\begin{align}
\label{eq:35_la1}&\int_{D^+_{R|a|}}|(i\nabla + A_0)\varphi_k^0|^2\,dx=O(|a|^2),\\
\label{eq:35_la2}
&\int_{\partial
  D^+_{R|a|}}|\varphi_k^0|^2\,ds=O(|a|^3)\quad\text{and}\quad
\int_{D^+_{R|a|}}|\varphi_k^0|^2\,dx=O(|a|^{4})
\quad\text{as }|a|\to0^+.
\end{align}
From estimates \eqref{eq:35_la1}--\eqref{eq:35_la2} and the Dirichlet
principle we deduce that 
\begin{align}
\label{eq:35vint_la}&\int_{D^+_{R|a|}}|(i\nabla + A_a) w_{k,R,a}^{int}|^2\,dx=O(|a|^2),\\
\label{eq:34vint_la}&\int_{\partial
  D^+_{R|a|}}|w_{k,R,a}^{int}|^2\,ds=O(|a|^3)\quad \text{and}\quad
\int_{D^+_{R|a|}}|w_{k,R,a}^{int}|^2\,dx=O(|a|^4)
\quad\text{as }|a|\to0^+.
\end{align}
For all  $R> 2$ and  $a=|a|p\in\Omega$ with
  $|a|$ small, we define 
 \begin{align}\label{eq:85_la} 
U_a^R(x):=\frac{w_{N,R,a}^{int} (|a|x)}{|a|^{j}}, \quad 
 W_a(x):=\frac{\varphi_N^0(|a|x)}{|a|^{j}}.
\end{align}
From \eqref{eq:asyphi0} we deduce that 
\begin{equation}\label{eq:vkext_la}
W_a\to  \beta e^{\frac i2\tilde \theta_0}\psi_j\quad\text{as } |a|\to0^+
\end{equation}
in $H^{1 ,0}(D^+_R,\C)$ for
every $R>2$, where $\psi_j$ is given in \eqref{eq:psi_j}
and $\beta\in\C\setminus\{0\}$ is as in \eqref{eq:asyphi0}.
Let  $u_R$ be the unique solution to the problem 
\begin{equation}\label{eq:equazione_w_R}
 \begin{cases}
  (i\nabla +A_{p})^2 u_R = 0, &\text{in }D^+_{R},\\
  u_R = e^{\frac{i}{2}(\theta_{p}-\theta_0^p)}e^{\frac i2\tilde\theta_0}
\psi_j, &\text{on }\partial D^+_{R}.
 \end{cases}
\end{equation}
Using the Dirichlet principle and \eqref{eq:vkext_la}, we can prove
that, for all $R>2$,
\begin{equation}\label{eq:vkint_la} 
  U_a^R \to  \beta u_R, \quad\text{in  }H^{1,p}(D_R^+,\C),
\end{equation}
as $a=|a|p\to 0$.

\begin{lemma}\label{l:conv_w_r}
For every $r>1$, $u_R\to\Psi_p$ in $H^{1,p}(D^+_r,\C)$ as $R\to+\infty$.  
\end{lemma}
\begin{proof}
  Let $r>2$. For every $R>r$, 
let $\eta_R:\R^2\to\R$ be a smooth cut-off function such
that $\eta_R\equiv 0$ in $D_{R/2}$,
  $\eta_R\equiv 1$ on $\R^2\setminus D_{R}$, $0\leq \eta_R\leq1$, and 
  $|\nabla\eta_R|\leq4/R$ in $\R^2$.
From the Dirichlet Principle, \eqref{eq:Psip3}, and
  \eqref{eq:Psip4} we deduce that 
  \begin{align*}
    \int_{D^+_{r}}& |(i\nabla +A_{p})(u_R-\Psi_p)|^2\,dx\leq
\int_{D^+_{R}} \left|(i\nabla +A_{p}) \left (
\eta_R(
 e^{\frac{i}{2}(\theta_{p}-\theta_0^p)}e^{\frac i2\tilde\theta_0}
\psi_j-\Psi_p) \right) \right|^2\,dx\\
&\leq  2\int_{\R^2_+\setminus D_{R/2}^+}
\left|(i\nabla +A_{p}) \left (
 e^{\frac{i}{2}(\theta_{p}-\theta_0^p)}e^{\frac i2\tilde\theta_0}
\psi_j-\Psi_p \right) \right|^2\,dx\\
&\quad +\frac{32}{R^2}
 \int_{D_{R}^+\setminus D_{R/2}^+}
\left|
 e^{\frac{i}{2}(\theta_{p}-\theta_0^p)}e^{\frac i2\tilde\theta_0}
\psi_j-\Psi_p \right|^2\,dx=o(1)
  \end{align*}
as $R\to+\infty$.
\end{proof}

\begin{lemma}\label{l:stima_Lambda0_sotto}
Let $p\in{\mathbb S}^1_+$. Let $\tilde R$ be as in Lemma \ref{l:stima_Lambda0_sopra}. 
 For all $R>\tilde R$ and $a=|a|p\in\Omega$ such that
 $|a|<\frac{R_0}{R}$, there holds
\[
\frac{\lambda_N -\lambda_N^a}{|a|^{2j}}\geq g_R(a)
\]
where $\lim_{|a|\to 0^+}g_R(a)=i|\beta|^2\tilde\kappa_R$,
being $\beta$ as in \eqref{eq:asyphi0} and 
\begin{equation}\label{eq:tildekappa_R}
\tilde\kappa_R=\int_{\partial D_R^+}\Big(
e^{-\frac{i}{2}(\theta_{p}-\theta_0^p)}e^{-\frac i2\tilde\theta_0}
(i\nabla+A_{p})u_R\cdot\nu
- (i\nabla) \psi_j\cdot\nu
\Big)\psi_j\,ds.
\end{equation}
\end{lemma}

\begin{proof}
Let  $\{\tilde w_{k,R,a}\}_{k=1,\dots,N}$ be the family of functions
resulting from $\{w_{k,R,a}\}$  
by the weighted Gram--Schmidt process
\[ 
\tilde w_{k,R,a}:= \dfrac{\hat w_{k,R,a}}{\sqrt{\int_\Omega q\,| \hat w_{k,R,a}|^2\,dx}}, \quad k=1,\ldots, N, 
\]
where $\hat w_{N,R,a} := w_{N,R,a}$ and, for $k=1,\ldots, N-1$,
$\hat w_{k,R,a} := w_{k,R,a} - \sum_{\ell= k+1}^{N} c_{\ell,k}^{R,a} \hat w_{\ell,R,a}$, 
with 
\[
 c_{\ell,k}^{R,a} :=
\dfrac{\int_\Omega
  q\, w_{k,R,a}  \overline{\hat w_{\ell,R,a}}\,dx}{
\int_\Omega
  q \,|\hat w_{\ell,R,a}|^2\,dx}.
 \]
By construction, there hold 
\begin{equation}\label{eq:gs-w}
\int_\Omega q|\tilde w_{k,R,a}|^2\,dx=1\text{ for all } 1 \leq k \leq N \quad\text{and}\quad
\int_\Omega q\, \tilde w_{k,R,a}\overline{\tilde w_{\ell,R,a}}\,dx=0\text{ for all }k\neq\ell.
\end{equation}
From \eqref{eq:orthonormality}, \eqref{eq:35_la2},
\eqref{eq:34vint_la}, \eqref{eq:vkext_la}, and \eqref{eq:vkint_la},
and an induction argument,it follows that
\begin{align}
&\label{eq:18w}
\int_\Omega
  q \,|\hat w_{k,R,a}|^2=1+O(|a|^{4})
\quad\text{and}\quad 
c_{\ell,k}^{R,a}=O(|a|^{4})\text{ for }\ell\neq k\quad\text{as }
|a|\to 0^+,\\
&\label{eq:13w}
\int_\Omega
  q \,|\hat w_{N,R,a}|^2\,dx=
\int_\Omega
  q \,|w_{N,R,a}|^2\,dx
=1+O\big(|a|^{2j+2}\big)
\quad\text{as }|a|\to 0^+,\\
\label{eq:25}&c_{N,k}^{R,a}= O\big(|a|^{3+j}\big)
\quad\text{as }|a|\to 0^+,\quad\text{for all }k<N.
\end{align}
From the classical 
Courant-Fisher minimax characterization of eigenvalues
and \eqref{eq:gs-w} it follows that 
\begin{equation*}
  \lambda_N^a \leq \max_
  {\substack{(\alpha_1,\dots, \alpha_{N})\in \C^{N}\\
      \sum_{k=1}^{{N}}|\alpha_k|^2 =1}}
  \int_{\Omega} \bigg|(i\nabla+A_a) \bigg(\sum_{k=1}^{{N}}\alpha_k
     \tilde w_{k,R,a}\bigg)\bigg|^2 dx,
\end{equation*}
so that 
\begin{equation}\label{eq:107w}
\lambda_N^a-\lambda_N\leq \max_
  {\substack{(\alpha_1,\dots, \alpha_{N})\in \C^{N}\\
      \sum_{k=1}^{{N}}|\alpha_k|^2 =1}}\sum_{k,n=1}^{N}
h_{k,n}^{a,R}\alpha_k \overline{\alpha_n},
\end{equation}
where 
\[
h_{k,n}^{a,R}= \int_{\Omega} (i\nabla+A_a)    \tilde w_{k,R,a}\cdot
\overline{(i\nabla+A_a)    \tilde w_{n,R,a}}\, dx
-\lambda_N \delta_{kn}.
\]
From \eqref{eq:13w}, \eqref{eq:85_la}, \eqref{eq:vkext_la}, and
\eqref{eq:vkint_la} it follows that 
\begin{align*}
 \notag h_{N,N}^{a,R} &= \dfrac{\lambda_N (1-
\int_\Omega
  q \,|w_{N,R,a}|^2\,dx)}{\int_\Omega
  q \,|w_{N,R,a}|^2\,dx} \\
&\notag\qquad+
  \dfrac{\left( \int_{D^+_{R|a|}} \big| (i\nabla +A_a)w_{N,R,a}^{int}
      \big|^2 dx
      - \int_{D^+_{R|a|}} \big| (i\nabla +A_0)\varphi_N^0 \big|^2 dx\right)}{\int_\Omega
  q \,|w_{N,R,a}|^2\,dx} \\
  \notag &=
   |a|^{2j}\bigg(\int_{D^+_R}|(i\nabla+A_p)U_a^R|^2\,dx-\int_{D^+_R}|(i\nabla+A_{0})W_a|^2\,dx+o(1)\bigg)\\
&\notag =
  |a|^{2j}|\beta|^2\bigg(\int_{D^+_R}|(i\nabla+A_p)u_R|^2\,dx-\int_{D^+_R}|\nabla\psi_j|^2\,dx+o(1)\bigg)\\
&=
-i  |a|^{2j}|\beta|^2(\tilde \kappa_R+o(1))
\end{align*}
as 
$|a|\to0^+$, with $\tilde\kappa_R$ as in \eqref{eq:tildekappa_R}.
From 
 \eqref{eq:18w}, \eqref{eq:25},
\eqref{eq:35vint_la}, and \eqref{eq:35_la1}, we obtain that, if $k<N$, 
\[
  h_{k,k}^{a,R}=(\lambda_k-\lambda_N)+o(1)\quad\text{as }|a|\to0.
\]
We observe that from \eqref{eq:simple_eigenv_n0} it follows that
$\lambda_k-\lambda_N<0$ for all $k<N$.

From 
\eqref{eq:vkext_la}, \eqref{eq:vkint_la}, \eqref{eq:35_la1}, \eqref{eq:35vint_la}, 
\eqref{eq:18w}, and \eqref{eq:13w}
we deduce that, for all $k<N$,
\[
  h_{k,N}^{a,R}=O(|a|^{1+j})\quad\text{and}\quad 
h_{N,k}^{a,R}=\overline{h_{k,N}^{a,R}}=O(|a|^{1+j})
\]
as 
$|a|\to 0^+$. Moreover, from  \eqref{eq:35_la1} and \eqref{eq:35vint_la} we have that, for all $k,n<N$ with $k\neq n$,
\[
  h_{k,n}^{a,R}=O(|a|^{2}) \quad\text{as }|a|\to0.
\]
Using \cite[Lemma 6.1]{abatangelo2015sharp} we can  conclude that 
\[
 \max_
  {\substack{(\alpha_1,\dots, \alpha_{N})\in \C^{N}\\
      \sum_{k=1}^{{N}}|\alpha_k|^2 =1}}\sum_{k,n=1}^{N}
h_{k,n}^{a,R}\alpha_k \overline{\alpha_n}=|a|^{2j}(-i  |\beta|^2\tilde \kappa_R+o(1))
\]
as $|a|\to 0^+$. The conclusion then follows from \eqref{eq:107w}.
\end{proof}

A combination of  Lemmas \ref{l:stima_Lambda0_sopra} and
\ref{l:stima_Lambda0_sotto} with Corollary \ref{l:Hbelow} yields the following preliminary
estimates of the eigenvalue variation.
\begin{corollary}\label{cor:preliminare_sottosopra}
  Let $p\in{\mathbb S}^1_+$.  Then
  \begin{enumerate}[\rm (i)]
  \item $|\lambda_N-\lambda_N^a|=O(1)\max\{H_a,|a|^{2j}\}$ as
    $a=|a|p\to0$;
\item $|\lambda_N-\lambda_N^a|=O(H_a^{{j}/{(j+\delta)}})$ as
    $a=|a|p\to0$.
  \end{enumerate}

\end{corollary}

\begin{proof}
As a direct consequence of Lemmas \ref{l:stima_Lambda0_sopra} and
\ref{l:stima_Lambda0_sotto}, we obtain that there exist $c_p,d_p\in\R$ such that, if $a=|a|p$ with $|a|$
sufficiently small, then
\begin{equation}\label{eq:cpdp}
c_p|a|^{2j}\leq \lambda_N-\lambda_N^a\leq d_p H_a.
\end{equation}
We notice that, up to now, we still do not have any indication of the
sign of the constants $c_p,d_p$. Estimate (i) follows directly from
\eqref{eq:cpdp}. Estimate (ii) follows combining (i) with \eqref{eq:stima_sotto_radiceH}.
\end{proof}

\begin{lemma} \label{l:lemmakappaR}
Let $\tilde\kappa_R$ be as in \eqref{eq:tildekappa_R}. Then,
\begin{align*} 
\lim_{R \to + \infty} \tilde\kappa_R = 2 i \mathfrak{m}_p,
\end{align*}
with $\mathfrak{m}_p$ as in \eqref{eq:m}.
\end{lemma}

\begin{proof}
First, for simplicity, we rename
\[
v_R = e^{-\frac{i}{2}(\theta_p - \theta_0^p)} e^{- \frac{i}{2}\tilde{\theta}_0} u_R,
\]
where $u_R$ is the unique solution of \eqref{eq:equazione_w_R}. Let's
introduce the function
\[
\varphi_R(r) = \int_{-\frac{\pi}{2}}^{\frac{\pi}{2}} v_R( r \cos t, r \sin t) \sin \left( j \left( \tfrac{\pi}{2} - t \right) \right) \, dt, \quad r > 1.
\]
By direct calculations, it is easy to verify that, since $-\Delta
v_R=0$ in $D_R^+\setminus D_1^+$, $\varphi_R$ satisfies
\begin{equation}\label{eq:22}
- \left( r^{1+2j} \left( r^{-j} \varphi_R(r) \right)' \right)' = 0, \quad \text{ for } r \in (1, R].
\end{equation}
Since $v_R = \psi_j$ on $\partial D_R^+$, we have that
\[
\varphi_R(R) = \frac{\pi}{2} R^j.
\]
Hence, by integrating \eqref{eq:22} over $(1,r)$, 
we get
\[
\varphi_R(r) = \frac{\frac{\pi}{2} - \varphi_R(1) R^{-2j}}{1 - R^{-2j}} r^j + \frac{\varphi_R(1) - \frac{\pi}{2}}{1 - R^{-2j}} r^{-j}, \quad r \in (1, R].
\]
By differentiation of the previous identity, we obtain that
\begin{align} \label{eq:varphiRprime}
\varphi_R'(R) = \frac{j R^{j-1}}{1 - R^{-2j}} \left( \frac{\pi}{2} (1 + R^{-2j}) - 2 \varphi_R(1) R^{-2j} \right).
\end{align}
On the other hand
\begin{align} \label{eq:varphiRprime2}
i \varphi_R'(R) = \frac{i}{R^{j+1}} \int_{\partial D_R^+} \frac{\partial v_R}{\partial \nu} \psi_j \, ds.
\end{align}
By combining \eqref{eq:varphiRprime} and \eqref{eq:varphiRprime2} we get
\begin{align} \label{eq:firstterm}
i \int_{\partial D_R^+} \frac{\partial v_R}{\partial \nu} \psi_j \, ds = \frac{i j}{1 - R^{-2j}} \left( \frac{\pi}{2} R^{2j} + \frac{\pi}{2} - 2 \varphi_R(1) \right).
\end{align}
The second term of the right hand side of \eqref{eq:tildekappa_R} can be calculated explicitly:
\begin{align} \label{eq:secondterm}
i \int_{\partial D_R^+} \frac{\partial \psi_j}{\partial \nu} \psi_j \, ds =i j \frac{\pi}{2} R^{2j}.
\end{align}
From \eqref{eq:firstterm}, \eqref{eq:secondterm} and \eqref{eq:tildekappa_R} it follows that
\begin{equation}\label{eq:24}
\tilde{\kappa}_R = \frac{i j}{1 - R^{-2j}} \left( - 2 \varphi_R(1) + \pi \right).
\end{equation}
Finally, Lemma \ref{l:conv_w_r} and Proposition \ref{prop:relationmp} imply that
\[
\lim_{R \to + \infty} \varphi_R(1) = \omega_p(1)+\frac\pi2=- \frac{\mathfrak{m}_p}{j} + \frac{\pi}{2}.
\]
This allows passing to the limit in \eqref{eq:24} thus getting the conclusion.
\end{proof}

\section{Energy estimates for the eigenfunction variation} \label{sec:energyestimates}

This section aims at providing some energy estimates for the
function $v_{N,R,a}$ defined in \eqref{eq:vNRa},
in order to improve the estimates on $H_a$ collected in Lemma
\ref{l:Hbelow}.

Throughout this section, we will regard the space
$H^1_0(\Omega,\C)$ (which coincides with $H^{1,0}_0(\Omega,\C)$, see \eqref{eq:equivalent_spaces}) as a real
Hilbert space endowed with the scalar product
\[
(u,v)_{H^{1,0}_{0,\R}(\Omega,\C)}=\Re\bigg(\int_\Omega (i\nabla+A_0) u
\cdot \overline{(i\nabla+A_0) v}\,dx\bigg),
\]
which induces on  $H^{1}_0(\Omega,\C)$ the norm \eqref{eq:norma} (with
$a=0$), which is equivalent to the Dirichlet norm, as observed in \eqref{eq:equivalent_spaces}.
To take in mind that here $H^{1}_0(\Omega,\C)$ is treated as a
vector space over $\R$, we denote it as $H^{1}_{0,\R}(\Omega,\C)$ and
its real dual space as  $(H^{1}_{0,\R}(\Omega,\C))^\star$.

Let us consider the function
\begin{align}\label{def_operatore_F}
 & F: \C \times H^{1}_{0,\R}(\Omega,\C)  \to  \R \times \R \times (H^{1}_{0,\R}(\Omega,\C))^\star\\
\notag  &F(\lambda,\varphi)= \Big( {\textstyle{
\|u\|_{H^{1,0}_0(\Omega,\C)}^2 -\lambda_N,\
      \mathfrak{Im}\big(\int_{\Omega}
      q(x)\varphi\overline{\varphi_N^0}\,dx\big), \ (i\nabla +A_0)^2
      \varphi-\lambda q \varphi}}\Big),
\end{align}
   where  $(i\nabla +A_0)^2  \varphi-\lambda \varphi\in (H^{1}_{0,\R}(\Omega,\C))^\star$ acts as 
\[
\phantom{a}_{(H^{1}_{0,\R}(\Omega,\C))^\star}\!\Big\langle (i\nabla
+A_0)^2 \varphi-\lambda q \varphi , u
\Big\rangle_{\!H^{1}_{0,\R}(\Omega,\C)}\!\!=\mathfrak{Re}
\left({\textstyle{\int_{\Omega}(i\nabla+A_0)\varphi\cdot\overline{(i\nabla+A_0)u}\,dx
      -\!\lambda \!\int_{\Omega} q \varphi\overline{u} \,dx}}\right)
\]
for all $\varphi\in H^{1}_{0,\R}(\Omega,\C)$.  In
\eqref{def_operatore_F} $\C$ is also meant as a vector space over
$\R$. From $(E_0)$ and \eqref{eq:9}, we have that $F(\lambda_N,\varphi_N^0)=(0,0,0)$.

\begin{lemma}\label{F_frechet}
  The function $F$ defined in \eqref{def_operatore_F}
  is Fr\'{e}chet-differentiable at $(\lambda_N,\varphi_N^0)$ and its
  Fr\'{e}chet-differential
$dF(\lambda_N,\varphi_N^0)\in \mathcal L\big( \C \times
H^{1}_{0,\R}(\Omega,\C),\R\times \R \times (H^{1}_{0,\R}(\Omega,\C))^\star\big)$
is invertible.
\end{lemma}
\begin{proof}
The proof follows from the Fredholm
alternative and assumption \eqref{eq:simple_eigenv_n0} by quite
standard arguments, see \cite[Lemma 7.1]{abatangelo2015sharp} for
details for a similar operator.
\end{proof}

\begin{theorem}\label{stima_teo_inversione}
Let $p\in{\mathbb S}^1_+$ and $R>\tilde R$, being $\tilde R$ as in
Lemma \ref{l:stima_Lambda0_sopra}. For $a=|a|p$ with $|a|<\tfrac{\bar
  r}{R}$, let $v_{N,R,a}$ be as defined in  \eqref{eq:vNRa}. 
 Then
$\|v_{N,R,a} -
\varphi_N^0\|_{H^{1,0}_0(\Omega,\C)}=O\left(\sqrt{H_a}\right)$
as $|a|\to0^+$.
\end{theorem}
\begin{proof}
From \eqref{eq:vNRa}, \eqref{def_blowuppate_normalizzate}, \eqref{eq:zar},
\eqref{eq:85_la}, we have that 
\begin{align*}
   \int_{\Omega}\big|(i\nabla+A_0)(v_{N,R,a}-\varphi_N^0)\big|^2\,dx=&
  \int_{\Omega}| e^{\frac{i}{2}(\theta_0^{{a}} - \theta_a)}
(i\nabla+A_a)\varphi_N^a-
  (i\nabla+A_0)\varphi_N^0|^2\,dx \\
  &\notag +
  H_a\int_{D_{R}^+}\Big|(i\nabla+A_0)\Big(Z_a^{R}-
  \tfrac{|a|^{j}}{\sqrt{H_a}} W_a
  \Big)\Big|^2\,dx\\
  &\notag -
  H_a\int_{D_{R}^+}\Big| e^{\frac{i}{2}(\theta_0^{p} -
    \theta_{p})} (i\nabla+A_{p})\tilde\varphi_a-\tfrac{|a|^{j}}{\sqrt{H_a}}
  (i\nabla+A_0)W_a \Big)\Big|^2\,dx.
\end{align*} 
We can estimate the second term at the right hand side in the
following way 
\begin{multline*}
  H_a\int_{D_{R}^+}\Big|(i\nabla+A_0)\Big(Z_a^{R}-
  \tfrac{|a|^{j}}{\sqrt{H_a}} W_a  \Big)\Big|^2\,dx \\\leq 2H_a \int_{D_{R}^+}\Big|(i\nabla+A_0)Z_a^{R}\Big|^2 
+ 2 |a|^{2j} \int_{D_{R}^+}\Big|(i\nabla+A_0)W_a\Big|^2  = O(|a|^2)
\end{multline*}
as $|a|\to0$, via \eqref{eq:stima_sopra_radiceH}, \eqref{eq:67_zar},  \eqref{eq:vkext_la}.
The estimate of the third term is analogous recalling  \eqref{eq:67} in addition.
In view of \eqref{eq:45}, we thus conclude
that $v_{N,R,a}\to \varphi_N^0$ in
$H^{1}_0(\Omega,\C)$ as $|a|\to 0^+$.
Therefore, we take advantage from Lemma \ref{F_frechet} and expand 
\begin{equation}\label{eq:37_bi}
 F(\lambda_N^a,v_{N,R,a}) = dF(\lambda_N,\varphi_N^0) (\lambda_N^a - \lambda_N, v_{N,R,a} - \varphi_N^0)
+ o\big(|\lambda_N^a - \lambda_N| + \|v_{N,R,a} -
\varphi_N^0\|_{H^{1,0}_0(\Omega,\C)}\big)
\end{equation}
 as $|a|\to 0$.
In view of Lemma \ref{F_frechet}, the operator $dF(\lambda_N,\varphi_N^0)$ is
invertible (and its inverse is continuous by the Open Mapping
Theorem), then from \eqref{eq:37_bi} it follows that 
\begin{multline*}
  |\lambda_N^a - \lambda_N| + \|v_{N,R,a} -
  \varphi_N^0\|_{H^{1}_0(\Omega,\C)}\\
  \leq\|(dF(\lambda_N,\varphi_N^0))^{-1}\|_{ \mathcal L( \R\times \R
    \times (H^{1}_{0,\R}(\Omega,\C))^\star,\C \times
    H^{1}_{0,\R}(\Omega,\C))} \| F(\lambda_N^a,v_{N,R,a})\|_{
    \R\times\R \times (H^{1}_{0,\R}(\Omega))^\star} (1+o(1))
\end{multline*}
as $|a|\to 0^+$.  It remains to estimate the norm of
\begin{align*}
  F&(\lambda_N^a,v_{N,R,a}) =\left( \alpha_a, \beta_a, w_a\right)
  \\
  \notag& = \left( \|v_{N,R,a}\|_{H^{1,0}_0(\Omega,\C)}^2
    -\lambda_N,
    \mathfrak{Im}\left({\textstyle{\int_{\Omega} q v_{N,R,a}\overline{\varphi_N^0}\,dx}}\right),
    (i\nabla+A_0)^2 v_{N,R,a} - \lambda_N^a\,q\, v_{N,R,a} \right)
\end{align*}
in $\R\times\R \times (H^{1}_{0,\R}(\Omega))^\star$.
As far as $\alpha_a$ is concerned, using \eqref{eq:67_zar},
\eqref{eq:67}, and Corollary
\ref{cor:preliminare_sottosopra} (part (ii)),  since $\delta<1\leq j$ we have that
\begin{align*}
 \alpha_a 
 &= \left(
 \int_{ D_{R|a|}^+} |(i\nabla+A_0)v_{N,R,a}^{int}|^2 \,dx-
 \int_{D_{R|a|}^+} |(i\nabla+A_a)\varphi_N^a|^2\,dx 
 \right) +(\lambda_N^a-\lambda_N)\\
 &=  H_a
\left(
 \int_{ D_{R}^+} |(i\nabla+A_0)Z_a^R|^2 \,dx-
 \int_{D_{R}^+} |(i\nabla+A_{p})\tilde\varphi_a|^2\,dx 
 \right) +(\lambda_N^a-\lambda_N)\\
&=O(H_a^{{j}/{(j+\delta)}})= O(\sqrt{H_a}),\quad\text{as }|a|\to0^+.
\end{align*}
As far as $\beta_a$ is concerned, by the normalization in \eqref{eq:6}, 
\eqref{eq:q_assumptions}, \eqref{eq:48zar}, \eqref{eq:48}, and \eqref{eq:asyphi0},
we have that 
\begin{align*}
  \beta_a &= \Im \left( \int_{D_{R|a|}^+} q\,
    v_{N,R,a}^{int}\overline{\varphi_N^0} \,dx- \int_{D_{R|a|}^+} q\,
    e^{\frac i2 (\theta_0^a-\theta_a)} \varphi_N^a\overline{\varphi_N^0}\,dx
    +\int_{\Omega} q \,e^{\frac i2 (\theta_0^a-\theta_a)}
    \varphi_N^a\overline{\varphi_N^0}\,dx
  \right)\\
  &=O( \sqrt{H_a} |a|^{j +2})=
  o\big(\sqrt{H_a}\big),\quad\text{as }|a|\to 0^+.
\end{align*}
Let $\varphi\in C^\infty_c(\Omega,\C)$. Then, if $|a|$ is
sufficiently small, $e^{\frac i2(\theta_a-\theta_0^a)}\varphi\in
H^{1,a}_0(\Omega,\C)$ and then, in view of \eqref{eq_eigenfunction}, 
\[
  0=\int_{\Omega}e^{\frac i2(\theta_0^a-\theta_a)}
  (i\nabla+A_a)\varphi_N^a\cdot\overline{(i\nabla+A_0) \varphi
  }\,dx-\lambda_N^a \int_{\Omega}q\,e^{\frac
    i2(\theta_0^a-\theta_a)} \varphi_N^a\overline{\varphi
  }\,dx.
\]
Hence, by \eqref{eq:vNRa},
\begin{align*}
\int_{\Omega}(i\nabla+A_0)&v_{N,R,a}\cdot\overline{(i\nabla+A_0)\varphi}\,dx
    -\lambda_N^a \int_{\Omega}q\, v_{N,R,a}\overline{\varphi}  \,dx\\
&=  
\int_{ D_{R|a|}^+} (i\nabla+A_0)  v_{N,R,a}^{int}\cdot\overline{(i\nabla+A_0) \varphi}\,dx
    -\lambda_N^a \int_{ D_{R|a|}^+}q  v_{N,R,a}^{int}\overline{\varphi}  \,dx\\
&-\int_{ D_{R|a|}^+}e^{\frac
    i2(\theta_0^a-\theta_a)} (i\nabla+A_a)\varphi_N^a\cdot\overline{(i\nabla+A_0) \varphi}\,dx
    +\lambda_N^a \int_{ D_{R|a|}^+}q\,e^{\frac
    i2(\theta_0^a-\theta_a)}\varphi_N^a\overline{\varphi}  \,dx,
  \end{align*}
  which, in view of \eqref{eq:46}, \eqref{eq:48}, \eqref{eq:46zar},
  \eqref{eq:48zar}, yields 
\[
\sup_{\varphi\in C^\infty_c(\Omega,\C)\setminus\{0\}}\frac{\phantom{a}_{(H^{1}_{0,\R}(\Omega,\C))^\star}\!\Big\langle (i\nabla
+A_0)^2 v_{N,R,a}-\lambda_N^a q v_{N,R,a} , \varphi
\Big\rangle_{\!H^{1}_{0,\R}(\Omega,\C)}}{\|\varphi\|_{H^{1,0}_0(\Omega,\C)}}=O(\sqrt{H_a}),
\]
as $|a|\to0^+$.
By density of $C^\infty_c(\Omega,\C)$ in $H^{1,0}_0(\Omega,\C)$ we
conclude that 
\[
  \|w_a\|_{(H^{1}_{0,\R}(\Omega,\C))^\star}=O\big(\sqrt{H_a}\big),\quad\text{as
  }|a|\to0^+,
\]
thus completing the proof.
\end{proof}

As a consequence of Theorem \ref{stima_teo_inversione}, we
obtain the following improvement of Corollary \ref{l:Hbelow}.
\begin{theorem}\label{conseguenza_stima_teo_inversione}
We have that $|a|^{2j}=O(H_a)$ as $a=|a|p \to0$.
\end{theorem}
\begin{proof}
  Directly from  scaling  and Theorem \ref{stima_teo_inversione}, we obtain 
  that, for every $R>\tilde R$,
  \begin{equation} \label{eq:scaling}
 \left(\int_{\big(\frac{1}{|a|}\Omega\big)\setminus D_{R}^+}\bigg|(i\nabla+A_{p})
 \Big(\tilde \varphi_a(x) - e^{\frac i2(\theta_{p}-\theta_0^p)}
\tfrac{|a|^{j}}{\sqrt{H_a}}W_a\Big)\bigg|^2dx\right)^{1/2}=O(1),
\quad\text{as }a=|a|p\to0,
\end{equation}
from which it follows that
\begin{equation*}
 \tfrac{|a|^{j}}{\sqrt{H_a}} \left(\int_{D_{2R}^+\setminus D_{R}^+}\bigg|(i\nabla+A_{p})
 \Big(e^{\frac i2(\theta_{p}-\theta_0^p)}W_a\Big)\bigg|^2dx\right)^{1/2} 
 \leq O(1) + \left(\int_{D_{2R}^+\setminus D_{R}^+}\bigg|(i\nabla+A_{p})\tilde \varphi_a(x)\bigg|^2dx\right)^{1/2} 
\end{equation*}
as $a=|a|p\to0$. Via \eqref{eq:vkext_la} and \eqref{eq:67}, this reads $\tfrac{|a|^j}{\sqrt{H_a}}=O(1)$ 
as $|a|\to0^+$, thus concluding the proof.
\end{proof}

\section{Blow-up analysis} \label{sec:blowup}

\begin{theorem}  \label{theo:exact_convergence}
For $p\in {\mathbb S}^1_+$ and  $a = |a| p \in \Omega$, 
let $\varphi_N^a$ solve \eqref{eq:equation_a}. 
Let $\tilde{\varphi}_a$ be as in \eqref{def_blowuppate_normalizzate}, 
$\bar{K}$ be as in \eqref{eq:26}, 
$\beta$ be as in \eqref{eq:asyphi0} and $\Psi_p$ be the function defined in \eqref{eq:Psip}. 
Then
\begin{equation}\label{eq:limit1}
 \lim_{a=|a|p \to 0} \frac{|a|^j}{\sqrt{H_a}} = \frac{1}{|\beta|}
 \sqrt{\frac{\bar{K}}{\int_{\partial D_{\bar{K}}^+} |\Psi_p|^2 \,ds}}
\end{equation}
and 
\begin{equation}\label{eq:limit2}
  \tilde{\varphi}_a \to \frac{\beta}{|\beta|}\sqrt{\frac{\bar{K}}{\int_{\partial D_{\bar{K}}^+} |\Psi_p|^2\,ds}} \Psi_p, \quad \text{ as } a=|a|p \to 0,
\end{equation}
in $H^{1,p}(D_R^+, \C)$ for every $R > 1$,  almost everywhere in
$\R^2_+$, and in $C^2_{\rm loc}(\overline{\R^2_+}\setminus\{p\},\C)$.
\end{theorem}
\begin{proof}
From Theorem \ref{conseguenza_stima_teo_inversione}, we know that
$\frac{|a|^j}{\sqrt{H_a}} = O(1)$ as $a=|a|p \to 0$.
Furthermore, we have that the family 
$\{ \tilde{\varphi}_a \, : \, a = |a| p, \, |a| < \frac{\bar{r}}{R} \}$ 
is bounded in $H^{1,p}(D_R^+, \C)$ for all $R \geq \bar{K}$, see \eqref{eq:67}. 
Then, by a diagonal process, for every sequence $a_n = |a_n| p$ with $|a_n| \to 0$, 
there exist $c \in [0, +\infty)$, 
$\tilde{\Phi} \in \bigcup_{R>1} H^{1 ,{p}}(D_R^+,\C)$, 
and a subsequence $a_{n_\ell}$ such that
\[
\lim_{\ell \to + \infty} \frac{|a_{n_\ell}|^j}{\sqrt{H_{a_{n_\ell}}}} = c
\]
and
\[
\tilde{\varphi}_{a_{n_\ell}} \rightharpoonup \tilde{\Phi} \quad \text{ weakly in } H^{1,p}(D_R^+, \C) \, \text{for every } R > 1 \, \text{ and almost everywhere}.
\]
By \eqref{def_blowuppate_normalizzate} and compactness of the trace
embedding, we have that 
\begin{align} \label{eq:boundary1}
\frac{1}{\bar{K}} \int_{\partial D_{\bar{K}}^+} |\tilde{\Phi}|^2 \, ds = 1;
\end{align}
in particular $\tilde{\Phi} \not\equiv 0$. Passing to the weak limit
in the equation satisfied by  $\tilde{\varphi}_a$, i.e. in equation
\begin{equation}\label{eq:10}
(i \nabla + A_p)^2 \tilde{\varphi}_a = \lambda_N^a |a|^2 q(|a|x)
\tilde{\varphi}_a, \quad \text{in }
 \tfrac{1}{|a|} \Omega = \{ x \in \mathbb{R}^2 \, : \, |a| x \in \Omega \},
\end{equation}
we obtain that $\tilde{\Phi}$ weakly solves
\begin{equation} \label{eq:tildePhi1}
(i \nabla + A_p)^2 \tilde{\Phi} = 0, \quad \text{ in } \mathbb{R}^2_+.
\end{equation}
By continuity of the trace operator $H^{1,p}(D_R^+, \C)\to
L^2(\{0\}\times(-R,R),\C)$ and vanishing of $\tilde{\varphi}_{a_{n_\ell}}$
on $\{0\}\times(-R,R)$ for large $\ell$, we also have that
\begin{align} \label{eq:tildePhi2} 
\tilde{\Phi} = 0, \quad \text{ on } \partial \mathbb{R}^2_+.
\end{align}
By elliptic estimates, we can prove that $\tilde{\varphi}_{a_{n_\ell}}
\to \tilde{\Phi}$ in $C^2_{\rm loc}(\overline{\R^2_+}\setminus\{p\},\C)$. Therefore, for every $R>1$,
$\int_{\partial D_R^+} |\tilde{\varphi}_{a_{n_\ell}}|^2 \, ds \to
\int_{\partial D_R^+} |\tilde{\Phi}|^2 \, ds$ as $\ell \to + \infty$ and, passing to the limit in
\eqref{eq:10} tested by $\tilde{\varphi}_{a_{n_\ell}}$, we
obtain that 
\[
\int_{D_R^+} |(i\nabla + A_p)\tilde{\varphi}_{a_{n_\ell}}|^2 \, dx  \to \int_{D_R^+} |(i\nabla + A_p) \tilde{\Phi}|^2 \, dx, \quad \text{ as } \ell \to + \infty.
\]
Therefore, in view of the Poincaré inequality \eqref{eq:Poincare},
we deduce the convergence of norms
$\| \tilde{\varphi}_{a_{n_\ell}} \|_{H^{1,p}(D_R^+,\C)} \to \|
\tilde{\Phi}  \|_{H^{1,p}(D_R^+,\C)}$ as $\ell \to + \infty$
and then conclude that the convergence $\tilde{\varphi}_{a_{n_\ell}}
\to \tilde{\Phi}$ is actually strong in $H^{1,p}(D_R^+,\C)$ for every $R > 1$.

Therefore we can pass to the limit along $a_{n_\ell}$ in
\eqref{eq:scaling} and, recalling \eqref{eq:vkext_la}, we obtain that 
\[
\int_{\mathbb{R}^2_+ \setminus D_R^+} \left| (i \nabla + A_p) \left( \tilde{\Phi} - e^{\frac{i}{2}(\theta_p - \theta_0^p + \tilde{\theta}_0)} c \beta \psi_j \right) \right|^2 \, dx < + \infty,
\]
for every $R >\tilde R$. 

This implies that $c > 0$;  indeed,
otherwise, $c = 0$ would imply that $\int_{\mathbb{R}^2_+} |(i \nabla
+ A_p) \tilde{\Phi}|^2 \, dx < + \infty$, 
which, in view of \eqref{eq:tildePhi1}-\eqref{eq:tildePhi2} and
\eqref{eq:hardy}, would yield  $\tilde{\Phi} \equiv 0$, thus contradicting \eqref{eq:boundary1}. Therefore, from \eqref{eq:tildePhi1}, \eqref{eq:tildePhi2} and Proposition \ref{prop_tildePsi} we have necessarily that
\begin{equation} \label{eq:tildePhi}
\tilde{\Phi} = c \beta \Psi_p.
\end{equation}
From \eqref{eq:tildePhi}, \eqref{eq:boundary1} and the fact that $c > 0$, we have that
\[
c= \frac{1}{|\beta|} \sqrt{\frac{\bar{K}}{\int_{\partial D_{\bar{K}}^+} |\Psi_p|^2 \, ds}},
\]
so that the convergences \eqref{eq:limit1}--\eqref{eq:limit2} hold
along the subsequence $\{a_{n_\ell}\}_\ell$.  Since the limits in
\eqref{eq:limit1}--\eqref{eq:limit2} depend neither on the sequence
$\{a_n\}_n$ nor the subsequence $\{a_{n_\ell}\}_\ell$, we conclude that
the convergences holds for $|a| \to 0^+$.
\end{proof}

\begin{proof}[Proof of Theorem \ref{thm:blow_up_varphiN}]
It follows directly from \eqref{eq:limit1}, \eqref{eq:limit2}.
\end{proof}

From Theorem \ref{theo:exact_convergence} it follows that  the
blow-up family of functions $Z_a^R$ introduced in \eqref{eq:zar}
converges to a multiple of the unique solution $z_R$ to
\begin{equation} \label{eq:zR}
\begin{cases}
 (i \nabla + A_0)^2 z_R = 0, & \text{in } D_R^+, \\
z_R = e^{\frac{i}{2}(\theta_0^p - \theta_p)} \Psi_p, & \text{on } \partial D_R^+.
\end{cases}
\end{equation}

\begin{lemma} \label{l:convzar}
Under the same assumptions as in Theorem \ref{theo:exact_convergence},
let $Z_a^R$ be as in \eqref{eq:zar}. Then, for all $R > \tilde R$,
\begin{equation*}
Z_a^R \to \frac{\beta}{|\beta|} \sqrt{\frac{\bar{K}}{\int_{\partial D_{\bar{K}}^+} |\Psi_p|^2 \, ds}} z_R, \quad \text{ in } H^{1,0}(D_R^+, \C),
\end{equation*}
as $a=|a|p \to 0$.
\end{lemma}
\begin{proof}
Once the convergence \eqref{eq:limit2} is established, it follows from
a standard  Dirichlet principle, see \cite[Lemma 8.3]{abatangelo2015sharp} for
details.
\end{proof}

\section{Sharp asymptotics for convergence of eigenvalues: $f_R(a)$} \label{sec:end}

In view of Lemmas \ref{l:stima_Lambda0_sopra} and
\ref{l:stima_Lambda0_sotto} and of the asymptotics of $H_a$ given by
\eqref{eq:limit1}, to compute the limit of $\frac{\lambda_N^a -
  \lambda_N}{|a|^{2j}}$ it remains to compute the limit
of $f_R(a)$ as $a=|a|p \to 0$ and $R \to + \infty$.

\begin{lemma} \label{l:limitfRa}
For all $R > \tilde{R}$ (where $\tilde{R}$ is given in Lemma \ref{l:stima_Lambda0_sopra}) and $a = |a| p \in \Omega$ with $|a| < \frac{R_0}{R}$, let $f_R(a)$ be as in Lemma \ref{l:stima_Lambda0_sopra}. Then,
\begin{equation*} 
\lim_{|a|\to 0^+} f_R(a) = -i \frac{\bar{K}}{\int_{\partial D_{\bar{K}}^+} |\Psi_p|^2 \, ds} \kappa_R,
\end{equation*}
where
\begin{equation} \label{eq:kappaR}
\kappa_R = \int_{\partial D_R^+} \left( (i\nabla + A_0) z_R \cdot \nu \, \overline{z_R} - (i \nabla + A_p) \Psi_p \cdot \nu \, \overline{\Psi_p} \right) \, ds.
\end{equation}
Furthermore,
$\lim_{R \to + \infty} \kappa_R = - 2 i \mathfrak{m}_p$, where
$\mathfrak{m}_p$ is defined in \eqref{eq:m}.  
\end{lemma}

\begin{proof}
First, we observe that, by Theorem \ref{theo:exact_convergence}, Lemma \ref{l:convzar}, and the equations of $z_R$ \eqref{eq:zR} and $\Psi_p$ \eqref{eq:Psip2},
\begin{align*}
& \lim_{|a| \to 0^+} f_R(a) = \lim_{|a| \to 0^+} \left( \int_{D_R^+} | (i \nabla + A_0) Z_a^R |^2 \, dx - \int_{D_R^+} | (i \nabla + A_p) \tilde{\varphi}_a |^2 \, dx \right) + o(1) \\
& = \frac{\bar{K}}{\int_{\partial D_{\bar{K}}^+} |\Psi_p|^2 \, ds} \left( \int_{D_R^+} | (i \nabla + A_0) z_R|^2 \, dx - \int_{D_R^+} | (i\nabla + A_p) \Psi_p|^2 \, dx \right) = - i \frac{\bar{K}}{\int_{\partial D_{\bar{K}}^+} |\Psi_p|^2 \, ds} \kappa_R,
\end{align*}
with $\kappa_R$ from \eqref{eq:kappaR}.
We divide the computation of the limit $\lim_{R \to + \infty} \kappa_R$ in two steps.

\textbf{Step 1.} We claim that
\begin{equation} \label{eq:kappaR2}
\kappa_R = \int_{\partial D_R^+} \left( e^{\frac{i}{2}(\theta_p - \theta_0^p)} (i \nabla + A_0) z_R  - (i \nabla + A_p) \Psi_p \right) \cdot \nu \, e^{- \frac{i}{2}(\theta_p - \theta_0^p + \tilde{\theta}_0)} \psi_j \, ds + o(1),
\end{equation}
as $R \to + \infty$. Indeed, we observe that $\kappa_R$ can be written
as
\begin{align*}
\kappa_R= \int_{\partial D_R^+} \left( e^{\frac{i}{2}(\theta_p - \theta_0^p)} (i \nabla + A_0) z_R  - (i \nabla + A_p) \Psi_p \right) \cdot \nu \, e^{- \frac{i}{2}(\theta_p - \theta_0^p + \tilde{\theta}_0)} \psi_j \, ds + I_1(R) + I_2(R),
\end{align*}
where
\begin{align*}
& I_1(R) = \int_{\partial D_R^+} (i \nabla + A_0) \left( z_R - e^{\frac{i}{2} \tilde{\theta}_0} \psi_j \right) \cdot \nu \, \left( \overline{ e^{\frac{i}{2} (\theta_0^p - \theta_p)} \Psi_p} - e^{- \frac{i}{2} \tilde{\theta}_0} \psi_j \right) \, ds, \\
& I_2(R) = - \int_{\partial D_R^+} (i \nabla + A_p) \left( \Psi_p - e^{\frac{i}{2} (\tilde{\theta}_0 + \theta_p - \theta_0^p)} \psi_j \right) \cdot \nu \, \left( \overline{\Psi_p} - e^{- \frac{i}{2}(\theta_p - \theta_0^p + \tilde{\theta}_0)} \psi_j \right) \, ds .
\end{align*}
Let $\eta_R$ be a smooth cut-off function  satisfying
\begin{equation*}
  \eta_R\equiv 0\text{ in }D_{R/2},\quad 
  \eta_R\equiv 1 \text{ on }\R^2\setminus D_{R},\quad 0\leq \eta_R\leq1\quad\text{and}\quad
  |\nabla\eta_R|\leq4/R\text{ in }\R^2.
\end{equation*}
 By testing the equation
\[
(i \nabla + A_p)^2 \left( \Psi_p - e^{\frac{i}{2} (\theta_p - \theta_0^p + \tilde{\theta}_0 ) } \psi_j \right) = 0,
\]
which is satisfied in $\mathbb{R}^2_+ \setminus D_{R}^+$, on $(\Psi_p
- e^{\frac{i}{2} (\theta_p - \theta_0^p + \tilde{\theta}_0 ) } \psi_j)
(1 - \eta_{2R})^2$, we obtain that
\begin{align*}
I_2(R) & = i \int_{\mathbb{R}^2_+ \setminus D_R^+} | (i \nabla + A_p) \left( \Psi_p - e^{\frac{i}{2} (\theta_p - \theta_0^p + \tilde{\theta}_0 ) } \psi_j \right) |^2 (1 - \eta_{2R})^2 \, dx \\
& + 2 \int_{\mathbb{R}^2_+ \setminus D_R^+} (1 - \eta_{2R})
 \left( \overline{\Psi_p} - e^{-\frac{i}{2} (\theta_p - \theta_0^p + \tilde{\theta}_0 ) } \psi_j \right) 
 (i \nabla + A_p) \left( \Psi_p - e^{\frac{i}{2} (\theta_p - \theta_0^p + \tilde{\theta}_0 ) } \psi_j \right) \cdot \nabla \eta_{2R} \, dx.
\end{align*}
Hence,
\begin{multline*}
  |I_2(R)|\leq 2 \int_{\mathbb{R}^2_+ \setminus D_R^+} \left|(i \nabla +
  A_p) \left( \Psi_p - e^{\frac{i}{2} (\theta_p - \theta_0^p +
      \tilde{\theta}_0 ) } \psi_j \right)\right|^2 \, dx\\ + \frac{4}{R^2}
  \int_{D_{2R}^+ \setminus D_R^+} \left| \Psi_p - e^{\frac{i}{2}
      (\theta_p - \theta_0^p + \tilde{\theta}_0 ) } \psi_j \right|^2
  \, dx \to 0,\quad\text{as $R \to + \infty$},
\end{multline*}
thanks to \eqref{eq:Psip3}, \eqref{eq:Psip4}. On the other hand, by testing the equation
$(i \nabla + A_0)^2 \big( z_R - e^{\frac{i}{2} \tilde{\theta}_0}  \psi_j  \big) = 0$
in $D_R^+$ on $\eta_R \big( e^{\frac{i}{2} (\theta_0^p - \theta_p)}
  \Psi_p - e^{\frac{i}{2}\tilde{\theta}_0} \psi_j \big)$, 
the Dirichlet principle yields that
\begin{align*}
|I_1(R)| & = \Big| i \int_{D_R^+} (i \nabla + A_0) \left( z_R - e^{\frac{i}{2} \tilde{\theta}_0} \psi_j \right) \cdot \overline{(i \nabla + A_0) \left( \eta_R \left( e^{\frac{i}{2} (\theta_0^p - \theta_p)} \Psi_p - e^{\frac{i}{2} \tilde{\theta}_0} \psi_j \right) \right)} \, dx \Big| \\
& \leq  \int_{D_R^+} \left| (i \nabla + A_0) \left( \eta_R \left( e^{\frac{i}{2}(\theta_0^p - \theta_p)} \Psi_p - e^{\frac{i}{2} \tilde{\theta}_0} \psi_j \right) \right) \right|^2 \, dx  \\
& \leq 2 \int_{D_R^+ \setminus D_{R/2}^+} \left| (i\nabla + A_0)
  \left(e^{\frac{i}{2}(\theta_0^p - \theta_p)} \Psi_p - e^{\frac{i}{2}
      \tilde{\theta}_0} \psi_j \right) \right|^2 \, dx \\
&\qquad+ \frac{32}{R^2} \int_{D_R^+ \setminus D_{R/2}^+} \left| e^{\frac{i}{2}(\theta_0^p - \theta_p)} \Psi_p - e^{\frac{i}{2} \tilde{\theta}_0} \psi_j \right|^2 \, dx.
\end{align*}
Hence $\lim_{R \to + \infty}I_1(R)=0$ thanks to \eqref{eq:Psip3} and
\eqref{eq:Psip4}. The proof of \eqref{eq:kappaR2} is thereby complete.

\textbf{Step 2.} We now compute $\lim_{R \to + \infty} \kappa_R$.
First, we define
\[
\zeta_R(r) = \int_{-\frac{\pi}{2}}^{\frac{\pi}{2}} e^{- \frac{i}{2}
  \tilde{\theta}_0(r \cos t, r \sin t)} z_R(r \cos t, r \sin t)
 \sin \left( j \left( \tfrac{\pi}{2} - t \right) \right) \, dt.
\]
Thanks to the equation satisfied by $z_R$ \eqref{eq:zR}, we have that
\[
\big(r^{1+2j} \big(r^{-j} \zeta_R(r) \big)' \big)' = 0, \quad \text{in } (0, R].
\]
Therefore, by integrating over $(r,R)$, we obtain, for some $B\in\C$,
\begin{align*}
\zeta_R(r) = \frac{\zeta_R(R)}{R^j} r^j - \frac{B}{R^{2j}} r^j + B r^{-j}, \quad \text{ in } (0, R].
\end{align*}
Next, we note that the function $z_R^0:= e^{-\tfrac{i}{2}\tilde\theta_0}z_R$ 
is a solution to $-\Delta z_R^0=0$ in $D_R^+$ and $z_R^0=0$ on $\partial \R^2_+ \cap D_R$, 
so $z_R^0=O(|x|)$ as $|x|\to0$ (see e.g. \cite{HW}). This implies that $B = 0$ and 
\[
 \zeta_R(r) = \frac{\zeta_R(R)}{R^j} r^j\quad\text{and}\quad \zeta_R'(r) =
 \frac{j \zeta_R(R)}{R^j} r^{j-1},
 \quad \text{ in } (0, R].
\]
On the other hand, we can compute
\begin{align*}
\zeta_R'(R) & =  \frac{1}{R^{j+1}} \int_{\partial D_R^+} \nabla \left( e^{- \frac{i}{2} \tilde{\theta}_0} z_R \right) \cdot \nu \, \psi_j \, ds = - \frac{i}{R^{j+1}} \int_{\partial D_R^+} (i \nabla + A_0) z_R \cdot \nu \, e^{- \frac{i}{2} \tilde{\theta}_0} \psi_j \, ds.
\end{align*}
Hence, by combining the two previous equations, we have that
\begin{equation} \label{eq:123} 
\int_{\partial D_R^+} (i \nabla + A_0) z_R \cdot \nu \, e^{- \frac{i}{2} \tilde{\theta}_0} \psi_j \, ds = i j R^j \zeta_R(R). 
\end{equation}
To compute explicitly $\zeta_R(R)$, we can use the boundary conditions
of $z_R$ on $\partial D_R^+$,  Proposition \ref{prop:relationmp} and
\eqref{eq:psi_j} to obtain 
\begin{align} \nonumber
  \zeta_R(R) & = \int_{-\frac{\pi}{2}}^{\frac{\pi}{2}} e^{\frac{i}{2} (\theta_0^p - \theta_p - \tilde{\theta}_0)(R \cos t, R \sin t)} \Psi_p (R \cos t, R \sin t) \, \sin \left( j \left( \tfrac{\pi}{2} - t \right) \right) \, dt \\
  & \label{eq:124}= \int_{- \frac{\pi}{2}}^{\frac{\pi}{2}} (w_p +
  \psi_j)(R \cos t, R \sin t) \, \sin \left( j \left( \tfrac{\pi}{2} -
      t \right) \right) \, dt = - \frac{\mathfrak{m}_p}{j R^j} + R^j
  \frac{\pi}{2}.
\end{align}
By combining \eqref{eq:123} and \eqref{eq:124}, we get
\begin{align} \label{eq:part1}
\int_{\partial D_R^+} (i \nabla + A_0) z_R \cdot \nu \, e^{- \frac{i}{2} \tilde{\theta}_0} \psi_j \, ds = - i \mathfrak{m}_p + i j R^{2j} \frac{\pi}{2}.
\end{align}
Next, in view of \eqref{eq:Psip} we rewrite
\[
\int_{\partial D_R^+} (i \nabla + A_p) \Psi_p \cdot \nu \, e^{ \frac{i}{2} (\theta_0^p - \theta_p -\tilde{\theta}_0)} \psi_j \, ds = i \int_{\partial D_R^+} \nabla (w_p + \psi_j) \cdot \nu \, \psi_j \, ds.
\]
By using Proposition \ref{prop:relationmp} and \eqref{eq:psi_j}, we immediately obtain that
\begin{align} \label{eq:part2}
i \int_{\partial D_R^+} \nabla (w_p + \psi_j) \cdot \nu \, \psi_j \, ds = i \mathfrak{m}_p + i j R^{2j} \frac{\pi}{2}.
\end{align}
Finally, by combining \eqref{eq:kappaR2}, \eqref{eq:part1} and
\eqref{eq:part2} we obtain that
$\lim_{R\to+\infty}\kappa_R=-2i{\mathfrak m}_p$, thus concluding the proof.
\end{proof}

\begin{proof}[Proof of Theorems \ref{t:main_straight1} and \ref{t:main_straight2}]
From Lemmas \ref{l:stima_Lambda0_sotto}, \ref{l:stima_Lambda0_sopra},
Theorem \ref{theo:exact_convergence} and Lemma \ref{l:limitfRa}, it
follows that, for all $R > \tilde{R}$,
\begin{align*}
i |\beta|^2 \tilde{\kappa}_R + o(1) & \leq \frac{\lambda_N - \lambda_N^a}{|a|^{2j}} \leq f_R(a) \frac{H_a}{|a|^{2j}} \\
& = \left(- i \frac{ \bar{K} }{ \int_{ \partial D_{\bar{K}}^+ } |\Psi_p|^2 \, ds } \kappa_R + o(1) \right) 
\left( |\beta|^2 \frac{ \int_{ \partial D_{\bar{K}}^+} |\Psi_p|^2 \, ds }{ \bar{K} } + o(1) \right),
\end{align*}
as $a=|a|p \to 0$. Hence,
\begin{align*}
i |\beta|^2 \tilde{\kappa}_R \leq \liminf_{a=|a|p\to 0}
\frac{\lambda_N - \lambda_N^a}{|a|^{2j}} \leq
 \limsup_{a=|a|p\to 0} \frac{\lambda_N - \lambda_N^a}{|a|^{2j}} \leq - i |\beta|^2 \kappa_R,
\end{align*}
for every $R > \tilde{R}$. From Lemmas \ref{l:lemmakappaR} and
\ref{l:limitfRa}, by letting $R \to + \infty$, we obtain that
\[
- 2 |\beta|^2 \mathfrak{m}_p \leq \liminf_{|a| \to 0^+} \frac{\lambda_N - \lambda_N^a}{|a|^{2j}} \leq \limsup_{|a|\to 0^+} \frac{\lambda_N - \lambda_N^a}{|a|^{2j}} \leq - 2 |\beta|^2 \mathfrak{m}_p,
\] 
which yields that 
\[
\lim_{a=|a|p\to 0} \frac{\lambda_N - \lambda_N^a}{|a|^{2j}}=- 2
|\beta|^2 \mathfrak{m}_p,
\]
thus proving \eqref{eq:38} together with Theorem
\ref{t:main_straight2}. Statements (i),(ii), and (iii) of Theorem
\ref{t:main_straight1} follow from  combination of Theorem
\ref{t:main_straight2}, Lemma \ref{lemma:mp_signs}, and Proposition \ref{p:continuity_m_p}.
\end{proof}

\appendix
 \section*{Appendix: Hardy \& Poincaré inequalities}
 \setcounter{section}{1}
 \setcounter{theorem}{0}
 \setcounter{equation}{0}

In this appendix we recall some well-known Hardy and Poincaré-type
inequalities used throughout the paper.

In 
\cite{LW99} the
following Hardy-type inequalities were proved:
\begin{equation}\label{eq:hardy_R2}
  \int_{\R^2} |(i\nabla+A_a)u|^2\,dx \geq \frac14 \int_{\R^2}\frac{|u(x)|^2}{|x-a|^2}\,dx,
\end{equation}
which holds for all functions $u\in{\mathcal D}^{1,2}_a(\R^2)$, being 
${\mathcal D}^{1,2}_a(\R^2)$ the completion of $C^\infty_c(\R^2
\setminus \{a\},\C)$ with respect to the norm
$\|(i\nabla+A_a)u\|_{L^2(\R^2,\C^2)}$,
and
\begin{equation}\label{eq:hardy}
  \int_{D_r(a)} |(i\nabla+A_a)u|^2\,dx \geq \frac14 \int_{D_r(a)}\frac{|u(x)|^2}{|x-a|^2}\,dx,
\end{equation}
which holds for all $r>0$, $a\in \R^2$  and $u\in H^{1,a}(D_r(a),\C)$,
see also \cite[Lemma 3.1 and Remark 3.2]{FFT2011}.

We also recall from \cite{noris2015aharonov} two Poincaré-type inequalities in half-balls.
\begin{lemma}[{\cite[Lemma 3.3]{noris2015aharonov}}] \label{lemma:Poincare}
Let $r>0$ and $a \in D_r^+$. For all $u \in H^{1,a}(D_r^+,\C)$, with $u=0$ on $\{x_1=0\}$, we have
\begin{align} \label{eq:Poincare}
\frac{1}{r^{2}} \int_{D_r^+} |u|^{2} \,dx \leq \frac{1}{r} \int_{\partial D_r^+} |u|^{2} \,ds + \int_{D_r^+} |(i \nabla + A_a) u|^2 \,dx.
\end{align}
\end{lemma}

\begin{lemma}[{\cite[Lemma 3.4]{noris2015aharonov}}] \label{lemma:Poincare_type}
Let $r>0$ and $a \in D_r^+$. For all $u \in H^{1,a}(D_r^+,\C)$, with $u=0$ on $\{x_1=0\}$, we have
\begin{align*} 
\frac{1}{r} \int_{\partial D_r^+} |u|^{2}\,ds \leq \int_{D_r^+} |(i \nabla + A_a) u|^{2}\,dx.
\end{align*}
\end{lemma}

\end{document}